\pdfoutput=1
\RequirePackage{ifpdf}
\ifpdf % We~are running pdfTeX in pdf mode
\documentclass[pdftex]{sigma}%,draft
\else
\documentclass{sigma}
\fi

\usepackage[all]{xy} %xypic macro for latex2.09

\usepackage{amscd}

\numberwithin{equation}{section}

\newtheorem{thm}{Theorem}[section]
\newtheorem{lem}[thm]{Lemma}
\newtheorem{cor}[thm]{Corollary}
\newtheorem{pro}[thm]{Proposition}

{ \theoremstyle{definition}
\newtheorem{ex}[thm]{Example}
\newtheorem{rmk}[thm]{Remark}
\newtheorem{defi}[thm]{Definition}
}

\newcommand{\be }{\begin{gather}}
\newcommand{\ee }{\end{gather}}

\newcommand{\g}{\mathfrak g}

\newcommand{\Real}{\mathbb R}

\newcommand{\huaA}{\mathcal{A}}
\newcommand{\huaB}{\mathcal{B}}
\newcommand{\huaC}{\mathcal{C}}
\newcommand{\huaF}{\mathcal{F}}
\newcommand{\huaL}{\mathcal{L}}
\newcommand{\huaO}{\mathcal{O}}
\newcommand{\huaT}{\mathcal{T}}
\newcommand{\huaH}{\mathcal{H}}
\newcommand{\huaX}{\mathcal{X}}

\newcommand{\frkL}{\mathfrak L}
\newcommand{\frkB}{\mathfrak B}
\newcommand{\frkX}{\mathfrak X}
\newcommand{\frkD}{\mathfrak D}
\newcommand{\frka}{\mathfrak a}
\newcommand{\frkg}{\mathfrak g}

\newcommand{\Courant}[1]{\left[\!\left[ #1\right]\!\right]}
\newcommand{\Courantt}[1]{\left[\hspace*{-3pt}\left[ #1\right]\hspace*{-3pt}\right]}
\newcommand{\half}{\frac{1}{2}}
\newcommand{\lon }{\,\rightarrow\,}
\newcommand{\CWM}{C^{\infty}(M)}

\newcommand{\Id}{\rm{Id}}

\newcommand{\dM}{\mathrm{d}}
\newcommand{\Hom}{\mathrm{Hom}}
\newcommand{\Sym}{\mathrm{Sym}}
\newcommand{\Der}{\mathrm{Der}}
\newcommand{\Ker}{\mathrm{ker}}
\newcommand{\ad}{\mathrm{ad}}
\newcommand{\Hi}{\mathrm{H}}
\newcommand{\MN}{\mathrm{MN}}
\newcommand{\Def}{\mathrm{def}}

\newcommand{\LP}{$\mathsf{LieRep}$~}

\begin{document}
\allowdisplaybreaks

\newcommand{\arXivNumber}{2108.08906}

\renewcommand{\PaperNumber}{054}

\FirstPageHeading

\ShortArticleName{Deformations and Cohomologies of Relative Rota--Baxter Operators on Lie Algebroids}

\ArticleName{Deformations and Cohomologies\\ of Relative Rota--Baxter Operators\\ on Lie Algebroids and Koszul--Vinberg Structures}

\Author{Meijun LIU~$^{\rm a}$, Jiefeng LIU~$^{\rm a}$ and Yunhe SHENG~$^{\rm b}$}

\AuthorNameForHeading{M.~Liu, J.~Liu and Y.~Sheng}

\Address{$^{\rm a)}$~School of Mathematics and Statistics, Northeast Normal University,\\
\hphantom{$^{\rm a)}$}~Changchun 130024, Jilin, China}
\EmailD{\href{mailto:liumj281@nenu.edu.cn}{liumj281@nenu.edu.cn}, \href{mailto:liujf534@nenu.edu.cn}{liujf534@nenu.edu.cn}}

\Address{$^{\rm b)}$~Department of Mathematics, Jilin University, Changchun 130012, Jilin, China}
\EmailD{\href{mailto:shengyh@jlu.edu.cn}{shengyh@jlu.edu.cn}}

\ArticleDates{Received February 02, 2022, in final form July 07, 2022; Published online July 13, 2022}

\Abstract{Given a Lie algebroid with a representation, we construct a graded Lie algebra whose Maurer--Cartan elements characterize relative Rota--Baxter operators on Lie algebroids. We~give the cohomology of relative Rota--Baxter operators and study infinitesimal deformations and extendability of order $n$ deformations to order $n+1$ deformations of relative Rota--Baxter operators in terms of this cohomology theory. We~also construct a graded Lie algebra on the space of multi-derivations of a vector bundle whose Maurer--Cartan elements characterize left-symmetric algebroids. We~show that there is a homomorphism from the controlling graded Lie algebra of relative Rota--Baxter operators on Lie algebroids to the controlling graded Lie algebra of left-symmetric algebroids. Consequently, there is a natural homomorphism from the cohomology groups of a relative Rota--Baxter operator to the deformation cohomology groups of the associated left-symmetric algebroid. As applications, we give the controlling graded Lie algebra and the cohomology theory of Koszul--Vinberg structures on left-symmetric algebroids.}


\Classification{53D17; 53C25; 58A12; 17B70}

\section{Introduction}%\label{sec:intr}

In this paper we use Maurer--Cartan elements to study deformations and cohomologies of relative Rota--Baxter operators on Lie algebroids. Applications are given to study deformations and cohomologies of Koszul--Vinberg structures on left-symmetric algebroids.

\subsection[Relative Rota--Baxter operators on Lie algebroids and Koszul--Vinberg structures]{Relative Rota--Baxter operators on Lie algebroids\\ and Koszul--Vinberg structures}

 The concept of Rota--Baxter operators on associative algebras was introduced by G.~Baxter~\cite{Ba} and G.-C.~Rota \cite{RotaI,RotaII} in the 1960s. It~also plays an important role in the Connes--Kreimer's algebraic approach~\cite{CK} to the renormalization in perturbative quantum field theory.
 In~\cite{Kuper1},
Kupershmidt introduced the notion of a relative Rota--Baxter operator (also called an $\huaO$-operator) on a Lie algebra in order to better understand the relationship between the classical Yang--Baxter equation and the related integrable systems. In addition, the defining relationship of a relative Rota--Baxter operator was also called the Schouten curvature in~\cite{Kos}. See \cite{Car,Fard,Gu0-1,Guo,Rota--Baxter_Lie_algebroid,PBG,Uchino2} for more details on relative Rota--Baxter operators and their applications.

 The notion of a Lie algebroid was introduced by Pradines in 1967, which
is a generalization of Lie algebras and tangent bundles.
See \cite{General-theory-of-Lie-groupoid-and-Lie-algebroid} for the general
theory about Lie algebroids. Relative Rota--Baxter operators (also called $\huaO$-operators) on Lie algebroids were introduced in \cite{LiuShengBaiChen} as a~method to construct left-symmetric algebroids. The notion of a left-symmetric algebroid is a~geometric generalization of a~left-symmetric algebra (also called pre-Lie algebras, see
the survey article~\cite{Pre-lie-algebra-in-geometry} for more details).
See \cite{BBo,LiuShengBaiChen,Boyom1,Boyom2} for more details and applications of left-symmetric algebroids.

 In \cite{lsb2}, motivated by the theory of Lie bialgebroids \cite{Lie-bialgebroid}, the notion of a left-symmetric bialgebroid was introduced as a geometric generalization of a left-symmetric bialgebra \cite{Left-symmetric-bialgebras}. The double of a left-symmetric bialgebroid is not a left-symmetric algebroid anymore, but a pre-symplectic algebroid~\cite{lsb}. This result is parallel to the fact that
 the double of a Lie bialgebroid is a Courant algebroid~\cite{lwx}. As a Poisson structure $\pi$ on a manifold gives rise to a Lie bialgebroid, a Koszul--Vinberg structure $H$ on a flat manifold gives rise to a left-symmetric bialgebroid. In particular, if the Koszul--Vinberg structure $H$ is nondegenerate, the inverse of $H$ is a pseudo-Hessian structure~\cite{Shima,Geometry-of-Hessian-structures} on a flat manifold.
 Therefore, Koszul--Vinberg structures and pseudo-Hessian structures are respectively symmetric analogues of Poisson structures and symplectic structures. See \cite{ABB,BBo, WLS} for recent studies on Koszul--Vinberg structures.

 \subsection{Deformations and cohomologies}
The theory of deformation plays a prominent role in mathematics and physics.
The idea of treating deformation as a tool to study the algebraic structures was introduced by Gerstenhaber in his work on associative algebras \cite{Gerstenhaber2,Gerstenhaber1} and then was extended to Lie algebras by Nijenhuis and Richardson \cite{NR1,NR3}. One remarkable
result in Poisson geometry is that M.~Kontsevich \cite{Kon} proved that every Poisson manifold has a deformation quantization. There is a well known slogan, often attributed to Deligne, Drinfeld and Kontsevich: every reasonable deformation theory is controlled by a differential graded Lie algebra, determined up to quasi-isomorphisms.

A suitable deformation theory of an algebraic structure can be summarized as the following general principle: on the one hand, for a given object with an algebraic structure, there should exist a differential graded Lie algebra whose Maurer--Cartan elements characterize deformations of this object. On the other hand, there should exist a suitable cohomology so that the infinitesimal of a formal deformation can be identified with a cohomology class, and then a theory of the obstruction to the integration of an infinitesimal deformation can be developed using this cohomology theory. It~is well-known that deformations of Poisson structures are controlled by the differential graded Lie algebra constructed by the Schouten--Nijenhuis bracket of multi-vector fields. Infinitesimal deformations and extendibility of order $n$ deformations of a Poisson structure are characterized in terms of the Poisson cohomology \cite{Hue, Lic}. There also exists a differential graded Lie algebra and a deformation cohomology given by M.~Crainic and I.~Moerdijk in \cite{deformation-cohomology-LA} on the space of multi-derivations which controls deformations of Lie algebroids. See \cite{FZ1,FZ2} for more details on simultaneous deformations of algebras and morphisms and their applications in Poisson geometry.

\subsection{Summary of the results and outline of the paper}
Since Koszul--Vinberg structures are symmetric analogues of Poisson structures, while there is a full developed deformation and cohomology theories for Poisson structures, it is natural to develop the deformation and cohomology theories for Koszul--Vinberg structures.
 Note that a~Koszul--Vinberg structure on a left-symmetric algebroid is a relative Rota--Baxter operator on its sub-adjacent Lie algebroid with respect to a certain representation (Proposition~\ref{pro:LSBi-H}). Thus we develop the deformation and cohomology theories for relative Rota--Baxter operators on Lie algebroids first.
Inspired by the construction of the differential graded Lie algebra controlling deformations of a relative Rota--Baxter operator on a Lie algebra in \cite{TBGS}, we construct a suitable differential graded Lie algebra that controls deformations of relative Rota--Baxter operators on Lie algebroids. See \cite{Das,TBGS} for more details on cohomologies and deformations of relative Rota--Baxter operators on Lie algebras and associative algebras. Following the idea of M. Crainic and I. Moerdijk in \cite{deformation-cohomology-LA}, we also construct a differential graded Lie algebra that controls deformations of a left-symmetric algebroid. There is a natural homomorphism from the controlling algebra of relative Rota--Baxter operators to the controlling algebra of left-symmetric algebroids.
\mbox{Using} the controlling algebra of relative Rota--Baxter operators on Lie algebroids, we construct a~differential graded Lie algebra whose Maurer--Cartan elements are Koszul--Vinberg structures. Consequently, we establish a cohomology theory for Koszul--Vinberg structures. We~hope that our study on Koszul--Vinberg structures will draw more attention to the geometry of Koszul--Vinberg structures.

The paper is organized as follows.
In Section \ref{sec:MC-RRB-operator}, first we construct a differential graded Lie algebra that controls deformations of relative Rota--Baxter operators on Lie algebroids. Then we give the cohomology theories of relative Rota--Baxter operators on Lie algebroids induced by this differential graded Lie algebra. In Section \ref{sce:Cohomology RB}, we give the cohomology of Rota--Baxter operators on Lie algebroids and analyze the cohomology of the Rota--Baxter operator on an action Lie algebroid. In Section \ref{sec:Cohomology-RRB-operator}, first we show that infinitesimal deformations of a relative Rota--Baxter operator are classified by the first cohomology group. Then for an order $n$ deformation, we define its obstruction class, which is a cohomology class in the second cohomology group, and show that an order $n$ deformation of a relative Rota--Baxter operator is extendable if and only if its obstruction class is trivial.
In Section \ref{sec:relation to MN bracket}, we construct a graded Lie algebra whose Maurer--Cartan elements are precisely left-symmetric algebroids. The deformation cohomology of left-symmetric algebroids can be given directly using this graded Lie algebra. We~show that there is a~homomorphism from the controlling graded Lie algebra of relative Rota--Baxter operators on Lie algebroids to the controlling graded Lie algebra of left-symmetric algebroids. Consequently, there is a natural homomorphism from the cohomology groups of a relative Rota--Baxter operator to the deformation cohomology groups of the associated left-symmetric algebroid.
In Section \ref{sec:Cohomology-KV-structure}, we give the deformation and cohomology theories of Koszul--Vinberg structures on left-symmetric algebroids as applications of the above general framework.

\subsection{Conventions and notations}
We~will adopt the following notations and conventions throughout the paper.
Let $i$, $j$ be positive integers. A permutation $\sigma$ of $\{1,2,\dots,i+j\}$ is called an $(i;j)$-{\it unshuffle} if $\sigma(1)<\cdots<\sigma(i)$ and $\sigma(i+1)<\cdots<\sigma(i+j)$. The set of all $(i;j)$-unshuffle will be denoted by $\mathbb{S}_{(i;j)}$. The notion of an $(i_1,\dots,i_k)$-unshuffle and the set $\mathbb{S}_{(i_1,\dots,i_k)}$ are defined analogously.

\section[Maurer--Cartan characterizations and cohomologies of relative Rota--Baxter operators on Lie algebroids]
{Maurer--Cartan characterizations and cohomologies \\of relative Rota--Baxter operators on Lie algebroids}\label{sec:MC-RRB-operator}

\subsection[The controlling algebra of relative Rota--Baxter operators on Lie algebroids]
{The controlling algebra of relative Rota--Baxter operators \\on Lie algebroids}

In this subsection, given a Lie algebroid with a representation we construct a graded Lie algebra whose Maurer--Cartan elements characterize relative Rota--Baxter operators on Lie algebroids. Consequently, we obtain the differential graded Lie algebra that controls deformations of a~rela\-tive Rota--Baxter operator.

\begin{defi}
A {\it Lie algebroid} structure on a vector bundle $\huaA\longrightarrow M$ is
a pair that consists of a Lie algebra structure $[\cdot,\cdot]_\huaA$ on
the section space $\Gamma(\huaA)$ and a bundle map
$a_\huaA\colon \huaA\longrightarrow TM$, called the anchor, such that the
following relation is satisfied:
\begin{gather*}
[x,fy]_\huaA=f[x,y]_\huaA+a_\huaA(x)(f)y,\qquad \forall f\in
\CWM,\quad x,y\in\Gamma(\huaA).
\end{gather*}
When the image of $a_\huaA$ is of constant rank, we call $\huaA$ a {\it regular Lie algebroid}.
\end{defi}

For a vector bundle $E\longrightarrow M$, we denote by $\frkD(E)$ the gauge Lie algebroid of the frame bundle $\huaF(E)$, which is also called the covariant differential operator bundle of $E$. See \cite{General-theory-of-Lie-groupoid-and-Lie-algebroid} for more details on the gauge Lie algebroid.

Let $(\huaA,[\cdot,\cdot]_\huaA,a_\huaA)$ and $(\huaB,[\cdot,\cdot]_\huaB,a_\huaB)$ be two Lie
algebroids (with the same base), a {\it base-preserving homomorphism}
from $\huaA$ to $\huaB$ is a bundle map $\varphi\colon \huaA\longrightarrow \huaB$ such
that
\begin{eqnarray*}
 a_\huaB\circ\varphi=a_\huaA,\qquad
 \varphi[x,y]_\huaA=[\varphi(x),\varphi(y)]_\huaB,\qquad \forall x,y\in\Gamma(\huaA).
\end{eqnarray*}

Recall that a {\it representation} of a Lie algebroid $\huaA$
 on a vector bundle $E$ is a base-preserving morphism $\rho$ form $\huaA$ to the Lie algebroid $\frkD(E)$.
Denote a representation by $(E;\rho).$
The {\it dual representation} of a Lie algebroid $\huaA$ on $E^*$ is the bundle map $\rho^*\colon \huaA\longrightarrow \frkD(E^*)$ given by
\begin{gather*}
\langle \rho^*(x)(\xi),u\rangle=a_\huaA(x)\langle \xi,u\rangle-\langle \xi,\rho(x)(u)\rangle,\qquad \forall x\in \Gamma(\huaA),\quad\xi\in\Gamma(E^*),\quad u\in\Gamma(E).
\end{gather*}
Given a representation $(E;\rho)$, the cohomology of $\huaA$ with coefficients in $E$ is the cohomology of the cochain complex
$(\oplus_{k=0}^{+\infty}C^k(\huaA,E),\partial_\rho )$, where $C^k(\huaA,E)=\Gamma(\Hom(\wedge^k \huaA, E))$ and the coboundary operator $\partial_\rho\colon C^k(\huaA,E)\to C^{k+1}(\huaA,E)$
is defined by
\begin{align*}
 \partial_\rho\varpi(x_1,\dots,x_{k+1}){}=&\sum_{i=1}^{k+1}(-1)^{i+1} \rho(x_i)\varpi(x_1,\dots,\hat{x_i},\dots,x_{k+1})
 \\
 &+\sum_{i<j}(-1)^{i+j}\varpi([x_i,x_j]_\huaA,x_1,\dots,\hat{x_i},\dots,\hat{x_j},\dots,x_{k+1}),
\end{align*}
for $\varpi\in C^k(\huaA,E) $ and $x_1,\dots,x_{k+1}\in \Gamma(\huaA)$.

\begin{defi}
A {\it \LP pair} is a pair of a Lie algebroid $(\huaA,[\cdot,\cdot]_\huaA,a_\huaA)$ and a representation $\rho$ of $\huaA$ on a vector bundle $E$. We~denote a \LP pair by $(\huaA,[\cdot,\cdot]_\huaA,a_\huaA;\rho)$, or simply by $(\huaA;\rho)$.
\end{defi}

\begin{defi}[\cite{LiuShengBaiChen}]
Let $(\huaA,[\cdot,\cdot]_\huaA,a_\huaA;\rho)$ be a \LP pair. A bundle map
 $T\colon E\longrightarrow \huaA$ is called a {\it relative Rota--Baxter operator} on a \LP pair $(\huaA,[\cdot,\cdot]_\huaA,a_\huaA;\rho)$ if
 \begin{gather*}
 [T(u),T(v)]_{\huaA}=T(\rho(T(u))(v)-\rho(T(v))(u)), \qquad\forall u,v\in\Gamma{(E)}.
\end{gather*}
\end{defi}

\begin{defi}
Let $(\frkg=\oplus_{k\in Z} \frkg_{k},[\cdot,\cdot],d)$ be a differential graded Lie algebra. An element $\theta\in \frkg_{1}$ is called a {\it Maurer--Cartan element} of $\frkg$ if it satisfies
\begin{gather*}
{\rm d}\theta+\frac {1}{2}[\theta,\theta]=0.
\end{gather*}
\end{defi}

In particular, a Maurer--Cartan element of a graded Lie algebra $(\frkg=\oplus_{k\in Z} \frkg_{k},[\cdot,\cdot])$ is an element $\theta\in \frkg_{1}$ satisfying $[\theta,\theta]=0$.

Let $(\huaA,[\cdot,\cdot]_\huaA,a_\huaA;\rho)$ be a \LP pair. Consider the graded vector space
\begin{gather*}
\huaC^*(E,\huaA)=\oplus_{k\geq0}\huaC^k(E,\huaA),\qquad
\text{where}\quad\huaC^k(E,\huaA):=\Gamma\big(\Hom\big({\wedge}^{k}E,\huaA\big)\big).
\end{gather*}
Now we give the controlling algebra of relative Rota--Baxter operators on Lie algebroids, which is the main tool in the following study.
\begin{thm}\label{thm:graded Lie algebra}
	For $P\in \huaC^m(E,\huaA)$ and $Q\in \huaC^n(E,\huaA)$, we define a bracket operation
\begin{gather}
\Courant{P,Q}(u_1,u_2,\dots,u_{m+n})\nonumber
\\ \qquad
{}=\sum_{\sigma\in\mathbb{S}_{(m,1,n-1)}} (-1)^\sigma P(\rho(Q(u_{\sigma(1)},\dots,u_{\sigma(m)}))
u_{\sigma(m+1)},u_{\sigma(m+2)},\dots,u_{\sigma(m+n)})\label{eq:graded Lie bracket}
\\ \qquad\hphantom{=}
{}-(-1)^{mn}\!\!\!\sum_{\sigma\in\mathbb{S}_{(n,1,m-1)}}\!\!\! (-1)^\sigma Q(\rho(P(u_{\sigma(1)},\dots,u_{\sigma(n)}))u_{\sigma(n+1)},u_{\sigma(n+2)},\dots,u_{\sigma(m+n)})\nonumber
\\ \qquad\hphantom{=}
{}+(-1)^{mn}\!\!\!\sum_{\sigma\in\mathbb{S}_{(n,m)}} \!\!\!(-1)^\sigma [P(u_{\sigma(1)}, u_{\sigma(2)},\dots,u_{\sigma(n)}),Q(u_{\sigma(n+1)},u_{\sigma(n+2)},\dots,u_{\sigma(m+n)})]_{\huaA},
\nonumber
\end{gather}
where $ u_1,u_2,\dots,u_{m+n} \in \Gamma(E)$.
Then $(\huaC^*(E,\huaA),\Courant{\cdot,\cdot})$ is a graded Lie algebra and its Maurer--Cartan elements are precisely relative Rota--Baxter operators on $(\huaA;\rho)$.
	\end{thm}
\begin{proof}
It~is straightforward to check that $\Courant{\cdot,\cdot}$ is skew-symmetric in all arguments and function linear. Thus $\Courant{P,Q}\in\huaC^{m+n}(E,\huaA)$ for all $P\in \huaC^m(E,\huaA)$ and $Q\in \huaC^n(E,\huaA)$, which implies that
$\Courant{\cdot,\cdot}$ is well defined.

It~was shown in \cite{TBGS} that the bracket $\Courant{\cdot,\cdot}$ provides a graded Lie algebra structure on the graded vector space $\oplus_{k\geq0}\Hom_\Real\big({\wedge}^{k}\Gamma(E),\Gamma(\huaA)\big)$. Thus $(\huaC^*(E,\huaA),\Courant{\cdot,\cdot})$ is a graded Lie algebra.

Let $T\colon E\rightarrow \huaA $ be a bundle map. By a direct calculation, we have
\begin{gather*}%\label{eq:T-Maurer--Cartan element}
\Courant{T,T}(u_1,u_2)=2(T(\rho(Tu_1)u_2)-T(\rho(Tu_2)u_1)-[Tu_1,Tu_2]_{\huaA}),\qquad
\forall u_1, u_2\in \Gamma(E).
\end{gather*}
Thus $T$ is a Maurer--Cartan element of the graded Lie algebra $(\huaC^*(E,\huaA),\Courant{\cdot,\cdot})$ if and only if $T$ is a relative Rota--Baxter operator on the \LP pair $(\huaA,[\cdot,\cdot]_\huaA,a_\huaA;\rho)$.
\end{proof}

Let $T\colon E\longrightarrow \huaA$ be a relative Rota--Baxter operator on the \LP pair $(\huaA,[\cdot,\cdot]_\huaA,a_\huaA;\rho)$. By~Theorem \ref{thm:graded Lie algebra}, $T$ is a Maurer--Cartan element of the graded Lie algebra $(\huaC^*(E,\huaA),\Courant{\cdot,\cdot})$. Note that $\tilde{\dM}_T:=\Courant{T,\cdot}$ is a graded derivation on the graded Lie algebra $(\huaC^*(E,\huaA),\Courant{\cdot,\cdot})$ satisfying $\tilde{\dM}_T^2=0$. Therefore, $(\huaC^*(E,\huaA),\Courant{\cdot,\cdot},\tilde{\dM}_T)$ is a differential graded Lie algebra.
\begin{thm}
Let $(\huaA,[\cdot,\cdot]_\huaA,a_\huaA;\rho)$ be a \LP pair and $T\colon E\longrightarrow \huaA$ a relative Rota--Baxter operator. Then for a bundle map $T'\colon E\longrightarrow \huaA$, $T+T'$ is still a relative Rota--Baxter operator on the \LP pair $(\huaA,[\cdot,\cdot]_\huaA,a_\huaA;\rho)$ if and only if $T'$ is a Maurer--Cartan element of the differential graded Lie algebra $\big(\huaC^*(E,\huaA),\Courant{\cdot,\cdot},\tilde{\dM}_T\big)$.
\end{thm}

\begin{proof}
{\sloppy	Assume that $T+T'$ is a relative Rota--Baxter operator on the \LP pair $(\huaA,[\cdot,\cdot]_\huaA,a_\huaA;\rho)$. By the fact that $T$ is a relative Rota--Baxter operator, we have
	\begin{gather*}
	\tilde{\dM}_T T'+\half\Courantt{T',T'}=\Courantt{T,T'}+\half\Courantt{T',T'}
	=\half\Courantt{T+T',T+T'}=0.
		\end{gather*}}\noindent
		Thus $T'$ is a Maurer--Cartan element of the differential graded Lie algebra $(\huaC^*(E,\huaA),\Courant{\cdot,\cdot},\tilde{\dM}_T)$.
		
The converse can be proved similarly. We~omit the details.
\end{proof}

\subsection{Cohomologies of relative Rota--Baxter operators on Lie algebroids}
In this subsection, we give a cohomology theory of relative Rota--Baxter operators on Lie algebroids, which will be used to study formal deformations of relative Rota--Baxter operators.

Let $T\colon E\longrightarrow \huaA$ be a relative Rota--Baxter operator on a \LP pair $(\huaA,[\cdot,\cdot]_\huaA,a_\huaA;\rho)$. Define $\dM_{T}\colon \huaC^k(E,\huaA)\rightarrow \huaC^{k+1}(E,\huaA)$ by
\begin{gather*}
\dM_{T}P=(-1)^{k}\tilde{\dM}_T P=(-1)^{k}\Courant{T,P},\qquad \forall P\in \huaC^k(E,\huaA).
\end{gather*}
Since $\tilde{\dM}_T\circ\tilde{\dM}_T=0 $, we have $\dM_{T}\circ \dM_{T}=0$. Thus $(\huaC^*(E,\huaA)=\oplus_{k\geq0}\huaC^k(E,\huaA),\dM_{T})$ is a cochain complex. Note the sign in the differential $\dM_T$ is motivated by Theorem \ref{pro:coboundary operator relation} below.

\begin{defi}\label{defi:cohomology of RB-operator}
The cochain complex $(\huaC^*(E,\huaA)=\oplus_{k\geq0}\huaC^k(E,\huaA),\dM_{T})$ is called
the {\it cohomology complex} of the relative Rota--Baxter operator $T$ on the \LP pair $(\huaA,[\cdot,\cdot]_\huaA,a_\huaA;\rho)$. The corresponding $k$-th cohomology group, denoted by $\huaH_{T}^k(E,\huaA)$, is called the {\it $k$-th cohomology group} for the relative Rota--Baxter operator $T$.
\end{defi}

We~give the coboundary operator $\dM_{T}$ explicitly.

\begin{pro}\label{pro:concret formula}
 For $P\in\huaC^k(E,\huaA)$ and $u_1,\dots,u_{k+1}\in\Gamma{(E)}$, we have
\begin{gather}
\dM_{T}P(u_1,u_2,\dots,u_{k+1})\nonumber
\\ \qquad
{}=\sum_{i=1}^{k+1}(-1)^{i+1}[Tu_i,P(u_1,u_2,\dots,\hat{u_i},\dots,u_{k+1})]_{\huaA}\nonumber
\\ \qquad\hphantom{=}
{}+\sum_{i=1}^{k+1}(-1)^{i+1}T\rho(P(u_1,u_2,\dots,\hat{u_i},\dots,u_{k+1}))(u_i)\nonumber
\\ \qquad\hphantom{=}
{}+\sum_{1\leq i<j\leq {k+1}}(-1)^{i+j}P(\rho(Tu_i)(u_j)-\rho(Tu_j)(u_i),u_1,\dots,\hat{u_i},\dots,\hat{u_j},\dots,u_{k+1}).
\label{eq:coboundary operator of RB-operator}
\end{gather}
\end{pro}

\begin{proof}
 It~follows from a direct calculation.
\end{proof}

It~is obvious that $P\in\huaC^1(E,\huaA)$ is closed if and only if
\begin{gather*}
[Tu,P(v)]_{\huaA}-[Tv,P(u)]_{\huaA}-T(\rho(P(u))(v)-\rho(P(v))(u))-P(\rho(Tu)(v)-\rho(Tv)(u))=0,
\end{gather*}
where $u,v\in\Gamma(E)$.

In the sequel, we give an alternative characterization of $\dM_{T}$ using the cohomology of Lie algebroids. First we recall a useful fact.

\begin{lem}[\cite{LiuShengBaiChen}]\label{lem:new Lie algebroid}
 Let $T\colon E\longrightarrow \huaA$ be a relative Rota--Baxter operator on a \LP pair $(\huaA,[\cdot,\cdot]_\huaA,a_\huaA;\rho)$. Then $(E,[\cdot,\cdot]_T,a_T=a_\huaA\circ T)$ is a Lie algebroid, where the bracket $[\cdot,\cdot]_T$ is given by
 \begin{gather*}
 [u,v]_T=\rho(T(u))v-\rho(T(v))u,\qquad\forall u,v\in\Gamma(E).
 \end{gather*}
 Furthermore, $T$ is a Lie algebroid homomorphism from $(E,[\cdot,\cdot]_T,a_T)$ to $(\huaA,[\cdot,\cdot]_\huaA,a_\huaA)$.
\end{lem}

Moreover, the Lie algebroid $(E, [\cdot,\cdot]_{T}, a_T)$ represents on the vector bundle $\huaA$.
\begin{lem}\label{ex:MC-ten}
 Let $T\colon E\longrightarrow \huaA$ be a relative Rota--Baxter operator on a \LP pair $(\huaA;\rho)$. Define $\varrho\colon E\lon\frkD(\huaA)$ by
\begin{gather*}%\label{eq:new rep}
\varrho(u)(x):=[Tu,x]_{\huaA}+T\rho(x)(u),\qquad x\in\Gamma{(\huaA)},\quad u\in\Gamma{(E)}.
\end{gather*}
Then $\varrho$ is a representation of the Lie algebroid $(E,[\cdot,\cdot]_{T},a_T=a_{\huaA}\circ T)$ on the vector bundle $\huaA$.
\end{lem}
\begin{proof}
By a direct calculation, we have
\begin{align*}
\varrho(fu)(x)&=[T(fu),x]_{\huaA}+T\rho(x)(fu)\\
&=f[Tu,x]_{\huaA}-a_{\huaA}(x)(f)(Tu)+fT\rho(x)(u)+Ta_{\huaA}(x)(f)u\\
&=f\varrho(u)(x)
\end{align*}
and
\begin{align*}
\varrho(u)(fx)&=[Tu,fx]_{\huaA}+T\rho(fx)(u)\\
&=f[Tu,x]_{\huaA}+a_{\huaA}(Tu)(f)(x)+fT\rho(x)(u)\\
&=f\varrho(u)(x)+a_{T}(u)(f)(x).
\end{align*}
It~is straightforward to check that $\varrho[u,v]_T=\varrho(u)\varrho(v)-\varrho(v)\varrho(u)$. Thus $\varrho$ is a representation of the Lie algebroid $(E,[\cdot,\cdot]_{T},a_T)$ on $\huaA$.
\end{proof}

\begin{rmk}
 Let $T\colon E\longrightarrow \huaA$ be a relative Rota--Baxter operator on a \LP pair $(\huaA;\rho)$. It~is straightforward to check that $(\huaA,E,\rho,\varrho)$ is a matched pair of Lie algebroids, where the Lie algebroid structure on $E$ is the Lie algebroid $(E,[\cdot,\cdot]_{T},a_T)$. See \cite{Mok} for more details on matched pairs of Lie algebroids.
\end{rmk}

{\sloppy\begin{thm}\label{pro:coboundary operator relation}
 Let $T\colon E\longrightarrow \huaA$ be a relative Rota--Baxter operator on a \LP pair $(\huaA,[\cdot,\cdot]_\huaA,a_\huaA;\rho)$. Then the coboundary operator of the relative Rota--Baxter operator $T$ is exactly the coboundary operator of the Lie algebroid $(E, [\cdot,\cdot]_{T}, a_T)$ with coefficients in the representation $(\huaA;\varrho)$, that is,
 $\dM_T=\partial_{\varrho}$.
\end{thm}}

\begin{proof}
By Proposition \ref{pro:concret formula}, for any $P\in\huaC^k(E,\huaA)$ and $u_1,\dots,u_{k+1}\in\Gamma{(E)}$, we have
\begin{gather*}
\dM_{T}P(u_1,u_2,\dots,u_{k+1})
\\ \qquad
{}=\sum_{i=1}^{k+1}(-1)^{i+1}[Tu_i,P(u_1,u_2,\dots,\hat{u_i},\dots,u_{k+1})]_{\huaA}
\\ \qquad\hphantom{=}
{}+\sum_{i=1}^{k+1}(-1)^{i+1}T\rho(P(u_1,u_2,\dots,\hat{u_i},\dots,u_{k+1}))(u_i)
\\ \qquad\hphantom{=}
{}+\sum_{1\leq i<j\leq {k+1}}(-1)^{i+j}P(\rho(Tu_i)(u_j)-\rho(Tu_j)(u_i),u_1,\dots,\hat{u_i},\dots,\hat{u_j},\dots,u_{k+1}),
\\ \qquad
{}=\sum_{i=1}^{k+1}(-1)^{i+1}\varrho(u_i)P(u_1,\dots,\hat{u_i},\dots,u_{k+1})
\\ \qquad\hphantom{=}
{}+\sum_{i<j}(-1)^{i+j}P([u_i,u_j]_T,u_1,\dots,\hat{u_i},\dots,\hat{u_j},\dots,u_{k+1})
\\ \qquad
{}=\partial_{\varrho}P(u_1,u_2,\dots,u_{k+1}).
\end{gather*}
The conclusion follows.
\end{proof}

\section{Cohomologies of Rota--Baxter operators on Lie algebroids}\label{sce:Cohomology RB}
In this section, first we give the cohomologies of Rota--Baxter operators on Lie algebroids with the help of the general framework of the cohomologies of relative Rota--Baxter operators. Then we study the cohomologies of the Rota--Baxter operator arising from an action of a Rota--Baxter Lie algebra on a manifold.

 Now we recall the notion of a Rota--Baxter operator on a Lie algebroid given in \cite{Rota--Baxter_Lie_algebroid}.
\begin{defi}
 A {\it Rota--Baxter operator} on a regular Lie algebroid $(\huaA,[\cdot,\cdot]_\huaA,a_\huaA)$ is a bundle map $R\colon \Ker(a_\huaA)\to \huaA$ such that
 \begin{gather*}%\label{eq:Rota--Baxter operator}
 [R(x),R(y)]_\huaA=R([R(x),y]_\huaA+[x,R(y)]_\huaA),\qquad \forall x,y\in\Gamma(\Ker(a_\huaA)).
 \end{gather*}
\end{defi}

For any $x\in\Gamma(\huaA)$, we define $\huaL_x\colon \Gamma(\huaA)\longrightarrow\Gamma(\huaA)$ by $\huaL_x(y)=[x,y]_\huaA$ for $y\in\Gamma(A)$. Then~$\huaL$ gives a representation of the Lie algebroid $\huaA$ on $ \Ker(a_\huaA)$. Thus a Rota--Baxter operator on a regular Lie algebroid $(\huaA,[\cdot,\cdot]_\huaA,a_\huaA)$ is a relative Rota--Baxter operator on the \LP pair $(\huaA,[\cdot,\cdot]_\huaA,a_\huaA;\huaL)$.

A {\it Rota--Baxter operator} on a Lie algebra $(\g,[\cdot,\cdot]_\g)$ is a linear map $\huaB\colon \g\to \g$ such that
\begin{gather*}
 [\huaB(u),\huaB(v)]_\g=\huaB([\huaB(u),v]_\g+[u,\huaB(v)]_\g),\qquad \forall u,v\in\g.
 \end{gather*}
The pair $(\g,\huaB)$ is called a {\it Rota--Baxter Lie algebra}.
\begin{rmk}
Since a vector space is a vector bundle over a point, a Lie algebra is naturally
a Lie algebroid with the anchor being zero. It~is not hard to see that a Rota--Baxter operator on a Lie algebroid reduces
to a Rota--Baxter operator on a Lie algebra when the underlying Lie algebroid reduces to a Lie algebra.
\end{rmk}

By Theorem \ref{thm:graded Lie algebra}, we have
\begin{cor}
Let $(\huaA,[\cdot,\cdot]_\huaA,a_\huaA)$ be a regular Lie algebroid. Then
\begin{itemize}\itemsep=0pt
 \item[$(i)$] $\bigl(\oplus_{k=0}^{\dim (\Ker(a_\huaA))} \Gamma\big(\Hom\big({\wedge}^{k}\Ker(a_\huaA),\huaA\big)\big),\Courant{\cdot,\cdot}\bigr)$ is a graded Lie algebra, where the graded Lie bracket $\Courant{\cdot,\cdot}$ is given by~\eqref{eq:graded Lie bracket}.
 \item[$(ii)$] $R$ is a Rota--Baxter operator on the regular Lie algebroid $\huaA$ if and only if $R$ is a Maurer--Cartan element of $\bigl(\oplus_{k=0}^{\dim (\Ker(a_\huaA))}\Gamma\big(\Hom\big({\wedge}^{k}\Ker(a_\huaA),\huaA\big)\big),\Courant{\cdot,\cdot}\bigr)$.
\end{itemize}
\end{cor}

By Lemmas \ref{lem:new Lie algebroid} and~\ref{ex:MC-ten}, we have

\begin{cor}
 Let $R\colon \Ker(a_\huaA)\longrightarrow \huaA$ be a Rota--Baxter operator on a regular Lie algebroid $(\huaA,[\cdot,\cdot]_\huaA,a_\huaA)$. Then $(\Ker(a_\huaA),[\cdot,\cdot]_R,a_R=a_\huaA\circ R)$ is a Lie algebroid, where the bracket $[\cdot,\cdot]_R$ is given by
\begin{gather*}
[u,v]_R:=[Ru,v]_\huaA+[u,Rv]_\huaA,\qquad \forall u,v\in\Gamma(\Ker(a_\huaA)).
\end{gather*}
 Furthermore, $ \mathbf{\varrho}\colon \Ker(a_\huaA)\lon\Der(\huaA)$ defined by
 \begin{gather*}
 \varrho(u)y:=[R(u),y]_\huaA-R[u,y]_\huaA, \qquad \forall y\in\Gamma(\huaA)
 \end{gather*}
 gives a representation of the Lie algebroid $(\Ker(a_\huaA),[\cdot,\cdot]_R,a_R)$ on the vector bundle $\huaA$.
\end{cor}

As a special case of Definition \ref{defi:cohomology of RB-operator}, we have
\begin{defi}%\label{defi:cochain complex of RB-operator}
 Let $R$ be a Rota--Baxter operator on a regular Lie algebroid $(\huaA,[\cdot,\cdot]_\huaA,a_\huaA)$. The cohomology of the cochain complex $\big({\oplus}_k\huaC^k(\Ker(a_\huaA),\huaA),\dM_R\big)$, where the coboundary operator $\dM_R\colon \huaC^k(\Ker(a_\huaA),\huaA)\to\huaC^{k+1}(\Ker(a_\huaA),\huaA)$ is given by \eqref{eq:coboundary operator of RB-operator} with $T=R$ and $\rho=\huaL$, is called the {\it cohomology of the Rota--Baxter operator} $R$. The corresponding $k$-th cohomology group, which we denote by $\huaH_{R}^k(\Ker(a_\huaA),\huaA)$, is called the {\it $k$-th cohomology group} for the Rota--Baxter operator~$R$.
\end{defi}

At the end of this section, we analyze the cohomology of the Rota--Baxter operator arising from an action of a Rota--Baxter Lie algebra on a manifold.

Let $(\frkg,[\cdot,\cdot]_{\frkg})$ be a Lie algebra. An {\it action}
of $\frkg$ on a manifold $M$ is a homomorphism
of Lie algebras $\phi \colon (\g, [\cdot,\cdot]_\g)\to(\mathfrak{X}(M),[\cdot,\cdot]_{\frkX(M)})$. For a Rota--Baxter operator $\huaB$ on a Lie algebra $(\g,[\cdot,\cdot]_\g)$, the bracket
{\sloppy\begin{gather*}%\label{eq:Rota--Baxter Lie bracket}
 [u,v]_\huaB=[\huaB(u),v]_\g+[u,\huaB(v)]_\g, \qquad \forall u,v\in\g
\end{gather*}
defines another Lie algebra structure on $\g$. Recall from \cite{Rota--Baxter_Lie_algebroid} that an action of a Rota--Baxter Lie algebra $(\g,\huaB)$
 on a manifold $M$ is a homomorphism
of Lie algebras $\phi \colon (\g, [\cdot,\cdot]_\huaB)\to(\mathfrak{X}(M),[\cdot,\cdot]_{\frkX(M)})$.
Let $\phi \colon (\g, [\cdot,\cdot]_\huaB)\to(\mathfrak{X}(M),[\cdot,\cdot]_{\frkX(M)})$ be an action of the Rota--Baxter Lie algebra $(\g,\huaB)$ on $M$. Consider the direct sum bundle $\huaA := (M \times\g)\oplus TM$. Then $\Gamma(\huaA) = (\CWM\otimes\g)\oplus\mathfrak{X}(M)$. There is naturally a Lie algebroid structure on $\huaA$ whose anchor~$a_{\huaA}$ is the projection to $TM$
and whose bracket is determined by
\begin{gather*}%\label{eq:action new bracket}
[fu+X,gv+Y]_{\huaA} := fg[u, v]_\g+X(g)v-Y(f)u+[X,Y]_{{\frkX(M)}},
\end{gather*}}\noindent
for all $X,Y\in\mathfrak{X}(M), u,v\in\g, f,g\in\CWM$. Consider the bundle map $R\colon \Ker(a_\huaA)=M\times\g\to (M\times\g)\oplus TM$ defined by
\begin{gather}\label{RB0}
 R(m,u):=(m,\huaB(u),\phi(u)(m)),\qquad \forall m\in M,\quad u\in\g.
\end{gather}
It~was proved in \cite{Rota--Baxter_Lie_algebroid}
 that the bundle map $R$ defined by $(\ref{RB0})$ is a Rota--Baxter operator on the Lie algebroid $((M \times\g)\oplus TM,[\cdot,\cdot]_{\huaA},a_{\huaA})$.

In the following, we establish the relations among the cohomology of the Rota--Baxter operator $R$ on the Lie algebroid $((M \times\g)\oplus TM,[\cdot,\cdot]_{\huaA},a_{\huaA})$, the cohomology of the Rota--Baxter operator $\huaB$ on the Lie algebra $\g$ and the cohomology of the Lie algebra homomorphism $\phi\colon (\g, [\cdot,\cdot]_\huaB)\to\mathfrak{X}(M)$.

The cohomology of a Rota--Baxter operator $\huaB$ on a Lie algebra $\g$ is the cohomology of the cochain complex $\big({\oplus}_{k=0}^{+\infty}\huaC^k(\g,\g),\dM_\huaB\big)$, where $ \huaC^k(\g,\g)=\Hom\big({\wedge}^k\g,\g\big)$ and the coboundary operator $\dM_\huaB\colon \Hom\big({\wedge}^k\g,\g\big)\lon \Hom\big({\wedge}^{k+1}\g,\g\big)$ is given by
\begin{gather*}
 \dM_\huaB f(u_1,\dots,u_{k+1})
 \\ \qquad
{}:=\sum_{i=1}^{k+1}(-1)^{i+1}[\huaB(u_i),f(u_1,\dots,\hat{u_i},\dots, u_{k+1})]_\g
\\ \qquad\hphantom{=}
{}+\sum_{i=1}^{k+1}(-1)^{i+1}\huaB[f(u_1,\dots,\hat{u_i},\dots, u_{k+1}),u_i]_\g
\\ \qquad\hphantom{=}
{}+\sum_{1\le i<j\le k+1}(-1)^{i+j}f([\huaB(u_i),u_j]_\g-[\huaB(u_j),u_i]_\g,u_1,\dots,\hat{u_i},\dots,\hat{u_j},\dots, u_{k+1}).%\label{eq:odiff}
\end{gather*}
See \cite{TBGS} for more details about the cohomology of Rota--Baxter operators on Lie algebras.

Let $\phi \colon (\g, [\cdot,\cdot]_\huaB)\to(\mathfrak{X}(M),[\cdot,\cdot]_{\frkX(M)})$ be an action of the Rota--Baxter Lie algebra $(\g,\huaB)$ on a manifold $M$. The cohomology of the Lie algebra homomorphism $\phi $ is the cohomology of the cochain complex
$\big({\oplus}_{k=0}^{+\infty}C_\phi^k(\g,\mathfrak{X}(M)),\dM_\phi\big)$, where $C_\phi^k(\g,\mathfrak{X}(M))=\Hom_\Real\big({\wedge}^k \g, \mathfrak{X}(M)\big)$ and the coboundary operator $\dM_\phi\colon C_\phi^k(\g,\mathfrak{X}(M))\to C_\phi^{k+1}(\g,\mathfrak{X}(M))$ is given by
\begin{align*}
\dM_\phi P(u_1,\dots,u_{k+1})={}&\sum_{i=1}^{k+1}(-1)^{i+1} [\phi(u_i),P(u_1,\dots,\hat{u_i},\dots,u_{k+1})]_{\mathfrak{X}(M)}
\\
&+\sum_{i<j}(-1)^{i+j}P([u_i,u_j]_\huaB,u_1,\dots,\hat{u_i},\dots,\hat{u_j},\dots,u_{k+1}),
\end{align*}
for $P\in C^k(\g,\mathfrak{X}(M)) $ and $u_1,\dots,u_{k+1}\in \g$. The corresponding $k$-th cohomology group, which we denote by $H_{\phi}^k(\g,\mathfrak{X}(M))$, is called the {\it $k$-th cohomology group} for the action of the Rota--Baxter Lie algebra $(\g,\huaB)$ on $M$. See \cite{Fregier,NR2} for more details on cohomology and deformations of Lie algebra homomorphisms.

The cochain complex associated to the Rota--Baxter operator $R$ on the Lie algebroid $((M \times\g)\allowbreak\oplus TM,[\cdot,\cdot]_{\huaA},a_{\huaA})$ defined by $(\ref{RB0})$ is given by
\begin{gather*}
\huaC^k(M\times \g,(M \times\g)\oplus TM)=\Gamma\big(\Hom\big({\wedge}^k(M\times \g),(M \times\g)\oplus TM\big)\big),\qquad k\geq0.
\end{gather*}
For any $k\geq 0$, define a linear map $\Xi\colon \huaC^k(\g,\g)\oplus C_\phi^k(\g,\mathfrak{X}(M))\to \huaC^k(M\times \g,(M \times\g)\oplus TM)$ by
\begin{gather}\label{eq:projection}
\Xi(P_1,P_2)(m,(u_1,\dots, u_k)):=(m,P_1(u_1,\dots, u_k),P_2(u_1,\dots, u_k)(m)),
\end{gather}
for all $m\in M$ and $u_1,\dots,u_k\in \g.$

\begin{pro}
	Let $\phi \colon (\g, [\cdot,\cdot]_\huaB)\to(\mathfrak{X}(M),[\cdot,\cdot]_{\frkX(M)})$ be an action of the Rota--Baxter Lie algebra $(\g,\huaB)$ on a manifold $M$. Then $\Xi$ defined by \eqref{eq:projection}, is a homomorphism of cochain complexes from $\big({\oplus}_{k\geq 0}\big(\huaC^k(\g,\g)\oplus C_\phi^k(\g,\mathfrak{X}(M))\big),(\dM_\huaB,\dM_\phi)\big)$ to $\big({\oplus}_{k\geq 0}\huaC^k(M\times \g,(M \times\g)\oplus TM),\dM_R\big)$, that is, $\Xi\circ(\dM_\huaB,\dM_\phi)=\dM_R\circ \Xi$. Consequently, $\Xi$ induces a homomorphism
\begin{gather*}
\Xi_\ast\colon\quad \huaH_{\huaB}^k(\g,\g)\oplus H_{\phi}^k(\g,\mathfrak{X}(M))\to \huaH_{R}^k(M\times \g,(M \times\g)\oplus TM),\qquad k\geq0,
\end{gather*}
	between the corresponding cohomology groups.
\end{pro}

\begin{proof}
	For any $P_1\in \huaC^k(\g,\g)$, $P_2\in C_\phi^k(\g,\mathfrak{X}(M))$ and $u_1,\dots,u_{k+1}\in \g$, we have
	\begin{gather*}
\dM_R\Xi(P_1,P_2)(1\otimes u_1,\dots, 1\otimes u_{k+1})
\\ \qquad
{}=\sum_{i=1}^{k+1}(-1)^{i+1} [R(1\otimes u_i),(P_1(u_1,u_2,\dots,\hat{u_i},\dots,u_{k+1}),P_2(u_1,\dots,\hat{u_i},\dots,u_{k+1}))]_{\huaA}
\\ \qquad\hphantom{=}
{}+\sum_{i=1}^{k+1}(-1)^{i+1} R[(P_1(u_1,u_2,\dots,\hat{u_i},\dots,u_{k+1}),P_2(u_1,\dots,\hat{u_i},\dots,u_{k+1})),u_i]_\huaA
\\ \qquad\hphantom{=}
{}+\sum_{1\leq i<j\leq {k+1}}(-1)^{i+j} (P_1([u_i, u_j]_\huaB, u_1,\dots,\hat{u_i},\dots,\hat{u_j},\dots,u_{k+1}),
\\
P_2([u_i,u_j]_\huaB,u_1,\dots,\hat{u_i},\dots,\hat{u_j},\dots,u_{k+1}))
\\ \qquad
{}=\Bigg(\sum_{i=1}^{k+1}(-1)^{i+1}[\huaB(u_i),P_1(u_1,u_2,\dots,\hat{u_i},\dots,u_{k+1})]_{\g}
\\ \qquad\hphantom{=}
{}+\sum_{i=1}^{k+1}(-1)^{i+1}\huaB[P_1(u_1,u_2,\dots,\hat{u_i},\dots,u_{k+1}),u_i]_\g
\\ \qquad\hphantom{=}
{}+\sum_{1\leq i<j\leq {k+1}}(-1)^{i+j}P_1([u_i, u_j]_\huaB,u_1,\dots,\hat{u_i},\dots,\hat{u_j},\dots,u_{k+1})
\\ \qquad\hphantom{=}	
{}+\sum_{i=1}^{k+1}(-1)^{i+1}[\phi(u_i),P_2(u_1,u_2,\dots,\hat{u_i},\dots,u_{k+1})]_{\mathfrak{X}(M)}
\\ \qquad\hphantom{=}
{}+\sum_{1\leq i<j\leq {k+1}}(-1)^{i+j}P_2([u_i, u_j]_\huaB,u_1,\dots,\hat{u_i},\dots,\hat{u_j},\dots,u_{k+1})\Bigg)
\\ \qquad
{}= (\dM_\huaB P_1	(u_1,\dots, u_{k+1}),\dM_\phi P_2	( u_1,\dots, u_{k+1}))
\\ \qquad	
{}=\Xi(\dM_\huaB P_1,\dM_\phi P_2)	(1\otimes u_1,\dots, 1\otimes u_{k+1}),
\end{gather*}
which implies that 	$\Xi\circ(\dM_\huaB,\dM_\phi)=\dM_R\circ \Xi$.
\end{proof}

\section[Formal deformations of relative Rota--Baxter operators on Lie algebroids]{Formal deformations of relative Rota--Baxter operators\\ on Lie algebroids}\label{sec:Cohomology-RRB-operator}

In this section, we use the cohomology of relative Rota--Baxter operators on Lie algebroids to study infinitesimal deformations and extendibility of order $n$ deformations of relative Rota--Baxter operators.

Let $(\huaA,[\cdot,\cdot]_\huaA,a_\huaA)$ be a Lie algebroid. The Lie algebroid structure on $\huaA$ can be extended to a Lie algebroid structure on $\huaA\otimes \Real[[t]]$ by replacing $\Real$-linearity of the bracket and anchor by $\Real[[t]]$-linearity and we denote it by $(\huaA\otimes \Real[[t]],[\cdot,\cdot]_\huaA,a_\huaA)$. For any representation $(E;\rho)$ of the Lie algebroid $\huaA$, there is also a natural representation of $\huaA\otimes \Real[[t]]$ on $E\otimes \Real[[t]]$ induced by $\rho$ and we also denote it by $\rho$.

\begin{defi}%\label{defi:formal deformation}
 A {\it formal deformation} of a relative Rota--Baxter operator $T\colon E\longrightarrow \huaA$ on a~\LP pair $(\huaA,[\cdot,\cdot]_\huaA,a_\huaA;\rho)$ is a formal power series
 \begin{gather*}
 T_t=\sum_{i=0}^{+\infty}\huaT_it^i\in \Hom(E,A) [[t]]
 \end{gather*}
 such that $T_t$ is a relative Rota--Baxter operator on the \LP pair $(\huaA\otimes \Real[[t]],[\cdot,\cdot]_\huaA,a_\huaA;\rho)$ and~$\huaT_0=T$.
\end{defi}

By a direct calculation, we see that $T_t$ is a relative Rota--Baxter operator on the \LP pair $(\huaA\otimes \Real[[t]],[\cdot,\cdot]_\huaA,a_\huaA;\rho)$ if and only if
\begin{gather}\label{eq:equivalent form}
\sum_{i+j=k}\!\!\big([\huaT_i(u),\huaT_j(v)]_\huaA-\huaT_j(\rho(\huaT_i(u))(v)-\rho(\huaT_i(v))(u))\big)=0,
\qquad \forall k\geq0,\quad u,v\in\Gamma(E).\!\!
\end{gather}

\begin{defi}%\label{defi: orde-n-deformation}
 Let $(\huaA,[\cdot,\cdot]_\huaA,a_\huaA;\rho)$ be a \LP pair and $T\colon E\longrightarrow \huaA$ a relative Rota--Baxter operator. If $T_{(n)}=\sum_{i=0}^{n}\huaT_it^i$ with $\huaT_0=T$, $\huaT_i\in\Hom(E,\huaA)$, $i=1,\dots,n$ is a relative Rota--Baxter operator on the \LP pair $\big(\huaA\otimes \Real[[t]]/\big(t^{n+1}\big),[\cdot,\cdot]_\huaA,a_\huaA;\rho\big)$, we say that $T_{(n)}$ is an {\it order~$n$ deformation} of the relative Rota--Baxter operator $T$. Furthermore, if there exists an element $\huaT_{n+1}\in \Hom(E,\huaA)$ such that $T_{(n+1)}=T_{(n)}+t^{n+1}\huaT_{n+1}$ is an order $n+1$ deformation of the relative Rota--Baxter operator $T$, we say that $T_{(n)}$ is {\it extendable}.
\end{defi}

An order $1$ deformation of a relative Rota--Baxter operator $T$ on a \LP pair $(\huaA,[\cdot,\cdot]_\huaA,a_\huaA;\rho)$ is called an {\it infinitesimal deformation} of the relative Rota--Baxter operator $T$.

\begin{defi}%\label{defi:equivalent and trivial}
Let $T\colon E\longrightarrow \huaA$ be a relative Rota--Baxter operator on $(\huaA,[\cdot,\cdot]_\huaA,a_\huaA;\rho)$. Two order $n$ deformations $T_t$ and $T'_t$ of $T$ are said to be {\it equivalent} if there is a formal series $\huaX_t=\sum_{i=1}^{+\infty}x_it^i$, $x_i\in \Gamma(\huaA) $ such that
\begin{gather}\label{eq:equivalent 2-1}
{\rm exp}(\ad_{\huaX_t})T_t=T_t' \mbox{ modulo } t^{n+1},
\end{gather}
 where ${\exp}$ denotes the exponential series and
\begin{gather*}
\ad^k_{\huaX_t}T_t=\big[\hspace{-3pt}\big[\huaX_t,\big[\hspace{-3pt}\big[\huaX_t,\dots,[\hspace{-1.5pt}[\huaX_t,\stackrel{}T_t
]\hspace{-1.5pt}],\stackrel{k}{\dots}\big]\hspace{-3pt}\big] \big]\hspace{-3pt}\big].
\end{gather*}
An order $n$ deformation $T_t$ of $T$ is called {\it trivial} if $T_t$ is equivalent to $T$.
\end{defi}

\begin{thm}\label{pro:equivalence class of O}
 There is a one-to-one correspondence between the equivalence classes of the infinitesimal deformations of a relative Rota--Baxter operator $T\colon E\longrightarrow \huaA$ on a \LP pair $(\huaA,[\cdot,\cdot]_\huaA,a_\huaA;\rho)$ and the first cohomology group $\huaH_{T}^1(E,\huaA)$.
\end{thm}
\begin{proof}
 For $k=1$ in \eqref{eq:equivalent form}, we have
 \begin{gather*}
 [Tu,\huaT_1 v]_\huaA\!-\![Tv,\huaT_1 u]_\huaA\!-\!T(\rho(\huaT_1 u)v\!-\!\rho(\huaT_1 v)u)\!-\!\huaT_1(\rho(T u)v\!-\!\rho(T v)u)\!=0,\quad\ \forall u,v\in\Gamma(\huaA),
 \end{gather*}
 which implies that $(\dM_T \huaT_1)(u,v)=0$, i.e., $\huaT_1$ is a $1$-cocycle.

 Assume that $T_t$ and $T'_t$ are equivalent infinitesimal deformations of the relative Rota--Baxter operator $T$. Comparing the coefficients of $t$ on both sides of \eqref{eq:equivalent 2-1} for $n=1$, we obtain
 \begin{eqnarray*}%\label{2-cocycle-equivalence3}
(\huaT_1'-\huaT_1)(u)=[Tu,x_1]_\huaA+T\rho(x_1)u=\dM_T x_1(u),
\end{eqnarray*}
which implies that
\begin{gather*}%\label{eq:deformation3}
\huaT_1'-\huaT_1=\dM_T x_1.	
\end{gather*}
Thus $\huaT_1$ and $\huaT'_1$ are in the same cohomology class.

The converse can be proved similarly. We~omit the details.
\end{proof}

It~is routine to check that
{\sloppy\begin{pro}\label{pro:trivial of O-operator}
Let $T\colon E\longrightarrow \huaA$ be a relative Rota--Baxter operator on a \LP pair $(\huaA,[\cdot,\cdot]_\huaA,a_\huaA;\rho)$ such that $\huaH_{T}^1(E,\huaA)=0$. Then all infinitesimal deformations of $T$ are tri\-vial.
\end{pro}}

\begin{thm}\label{thm:extendable of O operrator}
 Let $T_{(n)}=\sum_{i=0}^{n}\huaT_it^i$ be an order $n$ deformation of a relative Rota--Baxter ope\-ra\-tor $T$ on a \LP pair $(\huaA,[\cdot,\cdot]_\huaA,a_\huaA;\rho)$. Define
 \begin{gather*}
 \Theta=\half\sum_{\substack{i+j=n+1\\ i,j\geq1}}\Courant{\huaT_i,\huaT_j}.
 \end{gather*}
 Then the $2$-cochain $\Theta$ is closed, i.e., $\dM_T \Theta=0$.

 Furthermore, $T_{(n)}$ is {extendable} if and only the cohomology class $[\Theta]$ in $\huaH^2(E,\huaA)$ is trivial.
\end{thm}
\begin{proof}
By a direct calculation, we have
\begin{gather*}%\label{eq:construction of rRB operator}
\Theta(u,v)=\sum_{\substack{i+j=n+1\\ i,j\geq 1}} \big(\huaT_j(\rho(\huaT_i(u))(v)+\rho(\huaT_i(v))(u))-[\huaT_i(u),\huaT_j(v)]_\huaA\big),\qquad
\forall u,v\in\Gamma(E).
\end{gather*}
It~is not hard to check that
\begin{eqnarray*}
 \Theta(u,fv)= \Theta(fu,v)=f\Theta(u,v).
\end{eqnarray*}
Thus $\Theta\in \huaC^2(E,\huaA)$. The rest follows directly from the fact that this deformation problem is controlled by the differential graded Lie algebra $\big(\huaC^*(E,\huaA),\Courant{\cdot,\cdot},\tilde{\dM}_T\big)$. See the book~\cite{LPV} for more details.
\end{proof}

The above results on infinitesimal deformations and order $n$ deformations of relative Rota--Baxter operators on Lie algebroids can be easily applied to Rota--Baxter operators on Lie algebroids. We~omit the details.

\section[The Matsushima--Nijenhuis bracket for left-symmetric algebroids]{The Matsushima--Nijenhuis bracket\\ for left-symmetric algebroids}\label{sec:relation to MN bracket}

In this section, we construct a graded Lie algebra whose Maurer--Cartan elements are left-symmetric algebroids and study the relation with the controlling graded Lie algebra of relative Rota--Baxter operators on Lie algebroids.

\subsection{The Matsushima--Nijenhuis bracket for left-symmetric algebroids}

Recall that a {\it left-symmetric algebra} is a pair $(\frkg,\ast_\frkg)$, where $\g$ is a vector space, and $\ast_\frkg\colon \g\otimes \g\longrightarrow\g$ is a bilinear multiplication
satisfying that for all $x,y,z\in \g$, the associator
\begin{gather*}%\label{eq:associator}
(x,y,z):= x\ast_\frkg(y\ast_\frkg z)-(x\ast_\frkg y)\ast_\frkg z
\end{gather*} is symmetric in $x$, $y$,
i.e.,
\begin{gather*}
(x,y,z)=(y,x,z),
\end{gather*}
or equivalently,
\begin{gather*}
x\ast_\frkg(y\ast_\frkg z)-(x\ast_\frkg y)\ast_\frkg z=y\ast_\frkg(x\ast_\frkg z)-(y\ast_\frkg x)\ast_\frkg z.
\end{gather*}

\begin{defi}[\cite{LiuShengBaiChen,Boyom1}]%\label{defi:left-symmetric algebroid}
A {\it left-symmetric algebroid} structure on a vector bundle
$\huaA\longrightarrow M$ is a pair that consists of a left-symmetric
algebra structure $\ast_\huaA$ on the section space $\Gamma(\huaA)$ and a~vector bundle morphism $a_\huaA\colon \huaA\longrightarrow TM$, called the anchor,
such that for all $f\in\CWM$ and $x,y\in\Gamma(\huaA)$, the following
conditions are satisfied:
\begin{itemize}\itemsep=0pt
\item[$(i)$]$~x\ast_\huaA(fy)=f(x\ast_\huaA y)+a_\huaA(x)(f)y,$
\item[$(ii)$] $(fx)\ast_\huaA y=f(x\ast_\huaA y)$.
\end{itemize}
We~usually denote a left-symmetric algebroid by $(\huaA,\ast_\huaA, a_\huaA)$.
\end{defi}
Any left-symmetric algebra is a left-symmetric algebroid over a point.

Let $(\huaA,\ast_\huaA, a_\huaA)$ be a left-symmetric algebroid. For any $x\in\Gamma(\huaA)$, we define
$L_x\colon \Gamma(\huaA)\longrightarrow\Gamma(\huaA)$ and $R_x\colon \Gamma(\huaA)\longrightarrow\Gamma(\huaA)$ by
\begin{gather*}%\label{eq:leftmul}
L_xy=x\ast_\huaA y,\qquad
R_xy=y\ast_\huaA x,\qquad
\forall y\in\Gamma(\huaA).
\end{gather*}
Condition $(i)$ in the above definition means that $L_x\in \frkD(\huaA)$.
Condition $(ii)$ means that the map $x\longmapsto L_x$ is
$C^\infty(M)$-linear. Thus, $L\colon \huaA\longrightarrow \frkD(\huaA)$ is a bundle
map. With the same notations, there are two maps $L_x,R_x\colon \Gamma(\huaA^*)\longrightarrow\Gamma(\huaA^*)$ given by
\begin{gather}
\langle L_x \xi,y\rangle=a_\huaA(x)\langle\xi,y\rangle-\langle \xi,L_x y\rangle,\nonumber
\\
\langle R_x \xi,y\rangle=-\langle \xi,R_x y\rangle,\qquad
\forall x,y\in\Gamma(\huaA),\quad \xi\in\Gamma(\huaA^*).\label{eq:dualLR}
\end{gather}

\begin{pro}[\cite{LiuShengBaiChen}]%\label{thm:sub-adjacent}
 Let $(\huaA,\ast_\huaA, a_\huaA)$ be a left-symmetric algebroid. Define a skew-symmetric bilinear bracket operation $[\cdot,\cdot]_\huaA$ on $\Gamma(\huaA)$ by
\begin{gather*}
 [x,y]_\huaA=x\ast_{\huaA} y-y\ast_{\huaA} x,\qquad
 \forall x,y\in\Gamma(\huaA).
\end{gather*}
Then, $(\huaA,[\cdot,\cdot]_\huaA,a_\huaA)$ is a Lie algebroid, denoted by
$\huaA^c$, called the {\it sub-adjacent Lie algebroid} of
 $(\huaA,\ast_\huaA,a_\huaA)$. Furthermore, $L\colon \huaA\longrightarrow \frkD(\huaA)$ gives a
 representation of the Lie algebroid $\huaA^c$.
\end{pro}

A connection $\nabla$ on $M$ is said to be {\it flat} if the torsion tensor and the curvature tensor of~$\nabla$ vanish identically. A manifold $M$ endowed with a flat connection $\nabla$ is called a {\it flat manifold}.

\begin{ex}\label{ex:main}Let $(M,\nabla)$ be a flat manifold. Then $(TM,\nabla,\Id)$ is a left-symmetric algebroid whose sub-adjacent Lie algebroid is
exactly the tangent Lie algebroid. We~denote this left-symmetric algebroid by $T_\nabla M$.
\end{ex}

\begin{defi}
Let $E$ be a vector bundle over $M$, a {\it multiderivation} of degree $n$
is a multilinear map $D\in \Hom(\Lambda^{n}\Gamma(E)\otimes
\Gamma(E),\Gamma(E))$, such that for all $f\in C^\infty(M)$ and
$u_i\in \Gamma(E)$, $i=1,2,\dots,n+1$, the following conditions are satisfied:
 \begin{gather*}
D(u_1,\dots,fu_i,\dots,u_{n},u_{n+1})=f D(u_1,\dots,u_i,\dots,u_{n+1}),\qquad i=1,\dots,n,
\\
D(u_1,\dots,u_{n},fu_{n+1})=f D(u_1,\dots,u_{n},u_{n+1})+\sigma_D(u_1,\dots,u_{n})(f)u_{n+1},
\end{gather*}
where $\sigma_D\in \Gamma(\Hom(\Lambda^{n}\Gamma(E),TM))$ is called the
{\it symbol}. We~will denote by $\Der^n(E)$ the space of multiderivations
of $n,~n\geq 0$.
\end{defi}
We~denote by $\Der^*(E)=\oplus_{m}\Der^m(E)$ the space of multiderivations on a vector bundle $E$.

\begin{rmk}
The terminology ``multiderivation'' is usually referred to skew-symmetric ope\-ra\-tors, like in Crainic--Moerdijk's deformation complex of a Lie algebroid given in \cite{deformation-cohomology-LA}. For convenience, we also use the terminology ``multiderivation'' for the above case. Note that this kind of operators also appeared under the name {\rm Der}-valued forms in~\cite{Vita}.	
	\end{rmk}

\begin{thm}\label{thm:graded-one}
 For $D_1\in \Der^m(E)$ and $D_2\in \Der^n(E)$, we define the Matsushima--Nijenhuis bracket $ [\cdot,\cdot]_{\MN}\colon \Der^m(E)\times \Der^n(E)\lon \Der^{m+n}(E)$ by
 \begin{gather*}%\label{eq:MN bracket}
 [D_1,D_2]_{\MN}=D_1\circ D_2-(-1)^{mn}D_2\circ D_1,
 \end{gather*}
 where
 \begin{gather*}
(D_1\circ D_2)(u_1,u_2,\dots,u_{m+n+1})
\\ \qquad
{}=\sum_{\sigma\in \mathbb{S}_{(m,1,n-1)}}(-1)^{\sigma}D_1 (D_2(u_{\sigma(1)},\dots,u_{\sigma(m+1)}),u_{\sigma(m+2)},\dots,u_{\sigma(m+n)},u_{m+n+1})
\\ \qquad\hphantom{=}
{}+(-1)^{mn}\!\!\!\!\!\sum_{\sigma\in \mathbb{S}_{(n,m)}}\!\!\!\!\!(-1)^{\sigma}D_1
(u_{\sigma(1)},\dots,u_{\sigma(n)},D_2(u_{\sigma(n+1)},u_{\sigma(n+2)},\dots,u_{\sigma(m+n)},u_{m+n+1})).
\end{gather*}
Then $(\Der^*(E),[\cdot,\cdot]_{\MN})$ is a graded Lie algebra.

Furthermore, $\pi\in\Der^1(E)$ defines a left-symmetric algebroid structure on $E$ if and only if $[\pi,\pi]_{\MN}=0$, that is, $\pi$ is a Maurer--Cartan element of the graded Lie algebra $(\Der^*(E),[\cdot,\cdot]_{\MN})$.
\end{thm}

\begin{proof}
First, we show that the space of multiderivations is closed under the Matsushima--Nijenhuis bracket. For $D_1\in \Der^m(E)$ and $D_2\in \Der^n(E)$, by a direct calculation, we have
\begin{gather*}
[D_1,D_2]_{\MN}(fu_1,u_2,\dots,u_{m+n+1})
\\ \qquad
{}=fD_1\circ D_2(u_1,u_2,\dots,u_{m+n+1})-(-1)^{mn}fD_2\circ D_1(u_1,u_2,\dots,u_{m+n+1})
\\ \qquad\hphantom{=}
{}+\!\!\!\!\sum_{\sigma\in \mathbb{S}_{(m-1,1,n-1)}}\!\!\!\!(-1)^{\sigma}\sigma_{D_2} (u_{\sigma(2)},\dots,u_{\sigma(m+1)})(f)D_1(u_1,u_{\sigma(m+2)},\dots,u_{\sigma(m+n)},u_{m+n+1})
\\ \qquad\hphantom{=}
{}+(-1)^{mn}\sum_{\sigma\in \mathbb{S}_{(n-1,1,m-1)}}(-1)^{\sigma}\sigma_{D_1} (u_{\sigma(2)},\dots,u_{\sigma(n+1)})(f)
\\ \qquad\hphantom{=+(-1)^{mn}}
{}\times D_2(u_1,u_{\sigma(n+2)},\dots,u_{\sigma(m+n)},u_{m+n+1})
\\ \qquad\hphantom{=}
{}-(-1)^{mn}\sum_{\sigma\in \mathbb{S}_{(n-1,1,m-1)}}(-1)^{\sigma}\sigma_{D_1} (u_{\sigma(2)},\dots,u_{\sigma(n+1)})(f)
\\ \qquad\hphantom{=+(-1)^{mn}}
{}\times D_2(u_1,u_{\sigma(n+2)},\dots,u_{\sigma(m+n)},u_{m+n+1})
\\ \qquad\hphantom{=}
{}-\!\!\!\!\sum_{\sigma\in \mathbb{S}_{(m-1,1,n-1)}}\!\!\!\!(-1)^{\sigma}\sigma_{D_2} (u_{\sigma(2)},\dots,u_{\sigma(m+1)})(f)D_1(u_1,u_{\sigma(m+2)},\dots,u_{\sigma(m+n)},u_{m+n+1})
\\ \qquad
{}=f[D_1,D_2]_{\MN}(u_1,u_2,\dots,u_{m+n+1}),
\end{gather*}
which implies that
\begin{gather*}
[D_1,D_2]_{\MN}(fu_1,u_2,\dots,u_{m+n+1})=f[D_1,D_2]_{\MN}(u_1,u_2,\dots,u_{m+n+1}).
\end{gather*}
It~is straightforward to check that $[D_1,D_2]_{\MN}$ is skew-symmetric with respect to its first $m+n$ arguments. Thus $[D_1,D_2]_{\MN}$ is $\CWM$-linear with respect to its first $m+n$ arguments.

By a direct calculation, we have
\begin{align*}
[D_1,D_2]_{\MN}(u_1,u_2,\dots,fu_{m+n+1})={}&f[D_1,D_2]_{\MN}(u_1,u_2,\dots,u_{m+n+1})
\\
&+\sigma_{[D_1,D_2]_{\MN}}(u_1,u_2,\dots,u_{m+n})(f)u_{m+n+1},
\end{align*}
where the symbol $\sigma_{[D_1,D_2]_{\MN}}$ is given by
\begin{gather*}
\sigma_{[D_1,D_2]_{\MN}}(u_1,u_2,\dots,u_{m+n})(f)
\\ \qquad
{}=\sum_{\sigma\in \mathbb{S}_{(m,1,n-1)}}(-1)^{\sigma}\sigma _{D_{1}}(D_{2}(u_{\sigma(1)},\dots,u_{\sigma(m+1)}),u_{\sigma(m+2)},\dots,u_{\sigma(m+n)}))(f)
\\ \qquad\hphantom{=}
{}+\sum_{\sigma\in \mathbb{S}_{(n,1,m-1)}}(-1)^{\sigma}\sigma_{D_{2}} (D_{1}(u_{\sigma(1)},\dots,u_{\sigma(n+1)}),u_{\sigma(n+2)},\dots,u_{\sigma(m+n)})(f)
\\ \qquad\hphantom{=}
{}+(-1)^{mn}\sum_{\sigma\in \mathbb{S}_{(m,n)}}(-1)^{\sigma}\sigma _{D_{1}}(u_{\sigma(1)},\dots,u_{\sigma(n)})(\sigma _{D_{2}}(u_{\sigma(n+1)},\dots,u_{\sigma(n+m)}))(f)
\\ \qquad\hphantom{=}
{}+\sum_{\sigma\in \mathbb{S}_{(m,n)}}(-1)^{\sigma}\sigma _{D_{2}}(u_{\sigma(1)},\dots,u_{\sigma(m)})(\sigma _{D_{1}}(u_{\sigma(m+1)},\dots,u_{\sigma(m+n)}))(f).
\end{gather*}
Thus $[D_1,D_2]_{\MN}\in\Der^{m+n}(E)$.

It~was shown in \cite{ChaLiv,Nijenhuis} that the Matsushima--Nijenhuis bracket provides a graded Lie algebra structure on the graded vector space $\oplus_{n\geq 1}\Hom_\Real\big(\Lambda^{n-1}\Gamma(E)\otimes\Gamma(E),\Gamma(E)\big)$. We~have shown that $\Der^*(E)$ is closed under the Matsushima--Nijenhuis bracket. Thus $(\Der^*(E),[\cdot,\cdot]_{\MN})$ is a graded Lie algebra.

For $\pi\in\Der^1(E)$, we have
\begin{gather*}
 \pi(fu_1,u_2)=f\pi(u_1,u_2),\qquad
 \pi(u_1,fu_2)=f\pi(u_1,u_2)+\sigma_\pi(u_1)(f)u_2,\qquad
 \forall u_1,u_2\in\Gamma(E).
\end{gather*}
Furthermore, by a direct calculation, we have
\begin{align*}
[\pi,\pi]_{\MN}(u_1,u_2,u_3)={}&2(\pi(\pi(u_1,u_2),u_3)-\pi(\pi(u_2,u_1),u_3)-\pi(u_1,\pi(u_2,u_3))
\\
&+\pi(u_2,(u_1,u_3))).
\end{align*}
Thus $(E,\pi,\sigma_\pi)$ is a left-symmetric algebroid if and only if $[\pi,\pi]_{\MN}=0$.
\end{proof}

\begin{rmk}
The cohomology of left-symmetric algebras first appeared in the unpublished paper of Y. Matsushima. Then A. Nijenhuis constructed a graded Lie bracket in \cite{Nijenhuis}, which produces the cohomology theory for left-symmetric algebras. Thus the aforementioned graded Lie bracket is usually called the Matsushima--Nijenhuis bracket.	
\end{rmk}	

Let $(E,\pi,\sigma_\pi)$ be a left-symmetric algebroid. By Theorem $\ref{thm:graded-one}$, we have $[\pi,\pi]_{\MN}=0$. Because of the graded Jacobi identity, we get a coboundary operator $\dM_{\Def}\colon \Der^{n-1}(E)\rightarrow \Der^n(E)$ defi\-ned~by
\begin{gather*}
\dM_{\Def}(D)=(-1)^{n-1}[\pi,D]_{\MN}, \qquad
\forall D\in\Der^{n-1}(E).
\end{gather*}

\begin{pro}%\label{pro:MN bracket cohomology}
For all $D\in\Der^{n-1}(E)$, we have
\begin{gather}
\dM_{\Def}D(u_1,u_2,\dots,u_{n+1})\nonumber
\\ \qquad
{} =\sum_{i=1}^{n}(-1)^{i+1}\pi(u_i,D(u_1,u_2,\dots,\hat{u_i},\dots,u_{n+1}))\nonumber
\\ \qquad\hphantom{=}
{}+\sum_{i=1}^{n}(-1)^{i+1}\pi(D(u_1,u_2,\dots,\hat{u_i},\dots,u_n,u_i), u_{n+1})\nonumber
\\ \qquad\hphantom{=}
{}-\sum_{i=1}^{n}(-1)^{i+1}D(u_1,u_2,\dots,\hat{u_i},\dots,u_n,\pi(u_i, u_{n+1}))\nonumber
\\ \qquad\hphantom{=}
{}+\sum_{1\leq i<j\leq {n}}(-1)^{i+j}D(\pi(u_i,u_j)-\pi(u_j,u_i),u_1,\dots,\hat{u_i},\dots,\hat{u_j},\dots,u_{n+1})
\label{eq:deformation complex}
\end{gather}
for all $\ u_i\in \Gamma(E)$, $i=1,2,\dots,n+1$ and $\sigma_{\dM_{\Def}D}$
is given by
\begin{align*}
\sigma_{\dM_{\Def}D}(u_1,u_2,\dots,u_n)&= \sum_{i=1}^{n}(-1)^{i+1}[\sigma_\pi(u_i),\sigma_{D}(u_1,u_2,\dots,\hat{u_i},\dots,u_{n})]_{\frkX(M)}
\\
&\phantom{=}+\!\!\!\sum_{1\leq i<j\leq n}\!\!\! (-1)^{i+j}\sigma_{D}(\pi(u_i,u_j)-\pi(u_j,u_i),u_1,\dots,\hat{u_i},\dots,\hat{u_j},\dots,u_n)
\\
&\phantom{=}+\sum_{i=1}^{n}(-1)^{i+1}\sigma_\pi(D(u_1,u_2,\dots,\hat{u_i},\dots,u_n,u_i)).
%\label{eq:simbol}
\end{align*}
\end{pro}

\begin{proof}
 It~follows from straightforward verification.
\end{proof}

\begin{defi}
The cochain complex $\big(\Der^*(E)=\bigoplus _{n\geq 0}\Der^n(E),\dM_{\Def}\big)$ is called
the {\it deformation complex} of the left-symmetric algebroid $E$. The corresponding $k$-th cohomology group, which we denote by $\Hi_{\Def}^k(E)$, is called the {\it $k$-th deformation cohomology group}.
\end{defi}

\begin{rmk}
 The coboundary operator $\dM_{\Def}$ given by \eqref{eq:deformation complex} is exactly the coboundary operator given in \cite{LiuShengBaiChen} in the study of deformations of left-symmetric algebroids. Here we give this coboundary operator $\dM_{\Def}$ intrinsically using the Matsushima--Nijenhuis bracket.
\end{rmk}

\subsection[Relations between the graded Lie algebra \protect{$(C\textasciicircum{}*(E,A),[[cdot,cdot]])$} and \protect{$(Der\textasciicircum{}*(E),[cdot,cdot]\_\{MN\})$}]
{Relations between the graded Lie algebra $\boldsymbol{(\huaC^*(E,\huaA),\Courant{\cdot,\cdot})}$\\ and $\boldsymbol{(\Der^*(E),[\cdot,\cdot]_{\MN})}$}

Let $(\huaA,[\cdot,\cdot]_\huaA,a_\huaA;\rho)$ be a \LP pair. We~define a bundle map $\Phi\colon \huaC^{k}(E,\huaA)\rightarrow \Hom_\Real(\Lambda^{k}\Gamma(E)\otimes
\Gamma(E),\Gamma(E))$ as follows: for $P\in\huaC^{k}(E,\huaA)$,
\begin{gather}\label{ex:MC-eleven}
\Phi(P)(u_1,\dots,u_k,u_{k+1})=\rho(P(u_1,\dots,u_k))u_{k+1}, \qquad
\forall u_1,\dots,u_{k+1}\in\Gamma(E).
\end{gather}

\begin{lem}%\label{pro:homomorphism}
With the above notations, $\Phi(P)\in \Der^{k}(E)$ and $\sigma_{\Phi(P)}=a_\huaA\circ P.$
\end{lem}
\begin{proof}
By the properties of the representation $\rho$, we have
\begin{align*}
\Phi(P)(fu_1,\dots,u_k,u_{k+1})&=\rho(P(fu_1,\dots,u_k))u_{k+1}
\\
&=f\rho(P(u_1,\dots,u_k))u_{k+1}
\\
&=f\Phi(P)(u_1,\dots,u_k,u_{k+1}).
\end{align*}
Since $\Phi(P)$ is skew-symmetric with respect to its first $k$ arguments, $\Phi(P)$ is $\CWM$-linear with respect to its first $k$ arguments.

Similarly, by a direct calculation, we have
\begin{align*}
\Phi(P)(u_1,\dots,u_k,fu_{k+1})&=\rho(P(u_1,\dots,u_k))(fu_{k+1})
\\
&=f\rho(P(u_1,\dots,u_k))(u_{k+1})+a_{\huaA}(P(u_1,\dots,u_k))(f)u_{k+1}
\\
&=f\Phi(P)(fu_1,\dots,u_k,u_{k+1})+\sigma_{\Phi(P)}(u_1,\dots,u_k)(f)u_{k+1}.
\end{align*}
Thus $\Phi(P)\in\Der^{k}(E)$.
\end{proof}

Recall from Theorems \ref{thm:graded Lie algebra} and~\ref{thm:graded-one} that $(\huaC^*(E,\huaA),\Courant{\cdot,\cdot})$ and $(\Der^*(E),[\cdot,\cdot]_{\MN})$ are graded Lie algebras whose Maurer--Cartan elements are relative Rota--Baxter operators and left-symmetric algebroids respectively.
\begin{thm}\label{thm:homomorphism graded Lie}
Let $(\huaA,[\cdot,\cdot]_\huaA,a_\huaA;\rho)$ be a \LP pair. Then $\Phi $ given by \eqref{ex:MC-eleven} is a homomorphism of graded Lie algebras from $(\huaC^*(E,\huaA),\Courant{\cdot,\cdot})$ to $(\Der^*(E),[\cdot,\cdot]_{\MN})$.
\end{thm}
\begin{proof}
On the one hand, for $P\in\huaC^{n}(E,\huaA)$, $Q\in \huaC^{m}(E,\huaA)$, we have
\begin{gather*}
\Phi(\Courant{P,Q})(u_1,u_2,\dots,u_{m+n+1})
\\ \qquad
{}=\rho(\Courant{P,Q}(u_1,\dots,u_{m+n}))u_{m+n+1}
\\ \qquad
{}=\sum_{\mathbb{S}_{(m,1,n-1)}}(-1)^{\sigma}
\rho(P(\rho(Q(u_{\sigma(1)},\dots,u_{\sigma(m)}))u_{\sigma(m+1)},u_{\sigma(m+2)}, \dots,u_{\sigma(m+n)}))u_{m+n+1}
\\ \qquad
{}-\!(-1)^{mn}\!\!\!\!\!\!\sum_{\mathbb{S}_{(n,1,m-1)}}\!\!\!\!\!(-1)^{\sigma}
\rho(Q(\rho(P(u_{\sigma(1)},\dots,u_{\sigma(n)}))u_{\sigma(n+1)},u_{\sigma(n+2)}, \dots,u_{\sigma(m+n)}))u_{m+n+1}
\\ \qquad
{}+\!(-1)^{mn}\!\!\sum_{\mathbb{S}_{(n,m)}}\!\!(-1)^{\sigma}\!
\rho([P(u_{\sigma(1)},\dots,u_{\sigma(n)}),Q(u_{\sigma(n+1)},u_{\sigma(n+2)}, \dots,u_{\sigma(m+n)})]_{\huaA})u_{m+n+1}
\\ \qquad
{}=\sum_{\mathbb{S}_{(m,1,n-1)}}(-1)^{\sigma}
\rho(P(\rho(Q(u_{\sigma(1)},\dots,u_{\sigma(m)}))u_{\sigma(m+1)},u_{\sigma(m+2)}, \dots,u_{\sigma(m+n)}))u_{m+n+1}
\\ \qquad
{}-\!(-1)^{mn}\!\!\!\!\!\!\sum_{\mathbb{S}_{(n,1,m-1)}}\!\!\!\!\!(-1)^{\sigma}
\rho(Q(\rho(P(u_{\sigma(1)},\dots,u_{\sigma(n)}))u_{\sigma(n+1)},u_{\sigma(n+2)}, \dots,u_{\sigma(m+n)}))u_{m+n+1}
\\ \qquad
{}+\!(-1)^{mn}\sum_{\mathbb{S}_{(n,m)}}\!(-1)^{\sigma}
\rho(P(u_{\sigma(1)},\dots,u_{\sigma(n)}))\rho(Q(u_{\sigma(n+1)},u_{\sigma(n+2)}, \dots,u_{\sigma(m+n)}))u_{m+n+1}
\\ \qquad
{}+\!(-1)^{mn}\!\!\sum_{\mathbb{S}_{(n,m)}}\!(-1)^{\sigma}
\rho(Q(u_{\sigma(n+1)},u_{\sigma(n+2)},\dots,u_{\sigma(m+n)}))\rho(P(u_{\sigma(1)}, \dots,u_{\sigma(n)}))u_{m+n+1}.
\end{gather*}
On the other hand, we have
\begin{gather*}
(\Phi(P)\circ\Phi(Q))(u_1,u_2,\dots,u_{m+n+1})
\\ \qquad
{}=\!\!\sum_{\mathbb{S}_{(m,1,n-1)}}\!\!(-1)^{\sigma}\rho(P(\rho(Q(u_{\sigma(1)},\dots,u_{\sigma(m)})) u_{\sigma(m+1)},u_{\sigma(m+2)},\dots,u_{\sigma(m+n)}))u_{m+n+1}
\\ \qquad
{}+(-1)^{mn}\!\!\sum_{\mathbb{S}_{(n,m)}}\!\!\!(-1)^{\sigma}\rho(P(u_{\sigma(1)},\dots,u_{\sigma(n)})) \rho(Q(u_{\sigma(n+1)},u_{\sigma(n+2)},\dots,u_{\sigma(m+n)}))u_{m+n+1},
\end{gather*}
and
\begin{gather*}
-(-1)^{mn}(\Phi(Q)\circ\Phi(P))(u_1,u_2,\dots,u_{m+n+1})
\\ \qquad
{}=\sum_{\mathbb{S}_{(n,1,m-1)}}(-1)^{\sigma}\rho(Q(\rho(P(u_{\sigma(1)},\dots,u_{\sigma(n)})) u_{\sigma(n+1)},u_{\sigma(n+2)},\dots,u_{\sigma(m+n)}))u_{m+n+1}
\\ \qquad
{}-(-1)^{mn}\!\!\sum_{\mathbb{S}_{(n,m)}}\!\!(-1)^{\sigma}
\rho(Q(u_{\sigma(n+1)},u_{\sigma(n+2)},\dots,u_{\sigma(m+n)})) \rho(P(u_{\sigma(1)},\dots,u_{\sigma(n)}))u_{m+n+1}.
\end{gather*}
Thus, we have
\begin{gather*}
\Phi(\Courant{P,Q})=\Phi(P)\circ\Phi(Q)-(-1)^{mn}\Phi(Q)\circ\Phi(P)=[\Phi(P),\Phi(Q)]_\MN,
\end{gather*}
that is, $\Phi$ is a homomorphism from $(\huaC^*(E,\huaA),\Courant{\cdot,\cdot})$ to $(\Der^*(E),[\cdot,\cdot]_{\MN})$.
\end{proof}

The following conclusion has been proved in \cite{LiuShengBaiChen} by a direct calculation. We~give an intrinsic proof.
\begin{cor}\label{cor:construction left-symmetrci algebroid}
{\sloppy
Let $T\colon E\longrightarrow \huaA$ be a relative Rota--Baxter operator on a \LP pair $(\huaA,[\cdot,\cdot]_\huaA,a_\huaA;\rho)$. Then $(E, \ast_{T}, a_T=a_\huaA\circ T)$ is a left-symmetric algebroid, where $\ast_{T}$ is given~by
\begin{gather*}
u\ast_{T}v=\rho(Tu)(v),\qquad \forall u,v\in\Gamma(E).
\end{gather*}}
\end{cor}

\begin{proof}
 Since $T$ is a relative Rota--Baxter operator on a \LP pair $(\huaA,[\cdot,\cdot]_\huaA,a_\huaA;\rho)$, by Theo\-rem~\ref{thm:graded Lie algebra}, we have
\begin{gather*}
\Courant{T,T}=0.
\end{gather*}
By Theorem \ref{thm:homomorphism graded Lie}, we have
\begin{gather*}
[\Phi(T),\Phi(T)]_{\MN}=0.
\end{gather*}
By Theorem \ref{thm:graded-one}, $\Phi(T)$ provides a left-symmetric algebroid structure on $E$. Note that
\begin{gather*}
u\ast_{T}v=\Phi(T)(u,v)=\rho(Tu)(v).
\end{gather*}
 Thus $(E, \ast_{T}, a_T=a_\huaA\circ T)$ is a left-symmetric algebroid.
\end{proof}

\begin{thm}%\label{thm:cochain complex homomorphism}
Let $T$ be a relative Rota--Baxter operator on a \LP pair $(\huaA;\rho)$. Then $\Phi$ given by $(\ref{ex:MC-eleven})$ is a homomorphism from the cochain complex $(\huaC^*(E,\huaA),\dM_{T})$ to $(\Der^*(E),\dM_{\Def})$, that is, $\dM_{\Def}\circ\Phi=\Phi\circ \dM_{T}$. Consequently, $\Phi$ induces a homomorphism $\Phi_{*}\colon \huaH^k(E,\huaA)\rightarrow \Hi_{\Def}^{k}(E)$ from the cohomology groups of the relative Rota--Baxter operator $T$ to the deformation cohomology groups of the induced left-symmetric algebroid $(E, \ast_{T}, a_T)$.
\end{thm}

\begin{proof}
By Theorem \ref{thm:homomorphism graded Lie}, we have
\begin{gather*}
\Phi(\Courant{P,Q})=[\Phi(P),\Phi(Q)]_\MN.
\end{gather*}
Note that the left-symmetric algebroid structure on $E$ is given by $\Phi(T)$. For $P\in\huaC^{k}(E,\huaA)$, we have
\begin{eqnarray*}
 \dM_{\Def}\Phi(P)=(-1)^k[\Phi(T),\Phi(P)]_{\MN}=\Phi\big((-1)^k\Courant{T,P}\big)=\Phi(\dM_T P),
\end{eqnarray*}
which implies that $\dM_{\Def}\circ\Phi=\Phi\circ \dM_{T}$. The rest is direct.
\end{proof}

At the end of this section, we show that a formal deformation of a relative Rota--Baxter operator induces a formal deformation of the associated left-symmetric algebroid.

Recall that a {\it formal deformation} of a left-symmetric algebroid $(\huaA,\ast_{\huaA},a_{\huaA})$ is a left-symmetric algebroid $(\huaA\otimes \Real[[t]],\ast_t,a_t)$ with power series
\begin{gather*}
 \ast_t =\sum_{i=0}^{+\infty}\mu_it^i\in \Der^1(\huaA)[[t]],\qquad
 a_t=\sum_{i=0}^{+\infty}\frka_i t^i \in \Hom(\huaA,TM)[[t]],
\end{gather*}
such that $(\huaA\otimes \Real[[t]],\ast_t,a_t)_{t=0}=(\huaA,\ast_{\huaA},a_{\huaA})$.

\begin{pro}
 Let $T_t$ be a formal deformation of the relative Rota--Baxter operator $T\colon E\allowbreak\longrightarrow \huaA$ on a \LP pair $(\huaA,[\cdot,\cdot]_\huaA,a_\huaA;\rho)$. Then $(E\otimes \Real[[t]],\ast_t,a_t=a_\huaA\circ T_t)$ is a formal deformation of the left-symmetric algebroid $(E, \ast_{T}, a_T)$ associated to the relative Rota--Baxter operator $T$, where
\begin{gather*}
u\ast_t v=\rho(T_t(u))v,\qquad\forall u,v\in\Gamma(E).
\end{gather*}
\end{pro}
\begin{proof}
Since $T_t$ is a formal deformation of the relative Rota--Baxter operator $T$, by Corollary~\ref{cor:construction left-symmetrci algebroid}, $(E\otimes \Real[[t]],\ast_t,a_t=a_\huaA\circ T_t)$ is a left-symmetric algebroid. Note that $(E\otimes \Real[[t]],\ast_t,a_t)_{t=0}\allowbreak=(E, \ast_{T}, a_T)$. Thus $(E\otimes \Real[[t]],\ast_t,a_t)$ is a formal deformation of the left-symmetric algebroid $(E, \ast_{T}, a_T)$.
\end{proof}

\section[Maurer--Cartan characterizations and cohomology of Koszul--Vinberg structures on left-symmetric algebroids]
{Maurer--Cartan characterizations and cohomology \\of Koszul--Vinberg structures on left-symmetric algebroids}\label{sec:Cohomology-KV-structure}
In this section, we apply the controlling graded Lie algebra associated to relative Rota--Baxter operators to construct a graded Lie algebra whose Maurer--Cartan elements are precisely Koszul--Vinberg structures. Then we use this graded Lie algebra to study deformations of Koszul--Vinberg structures.

\subsection{Maurer--Cartan characterizations of Koszul--Vinberg structures}
Let us first recall the cochain complex of a left-symmetric algebroid with coefficients in the trivial representation. See \cite{LiuShengBaiChen} for the general theory of cohomology of left-symmetric algebroids. Let $(\huaA,\ast_\huaA,a_\huaA)$ be a left-symmetric algebroid. The set of $n$-cochains is given by
\begin{gather*}
C^{n}(\huaA)=\Gamma\big({\wedge}^{n-1}\huaA^*\otimes \huaA^*\big),\qquad
n\geq1.
\end{gather*}
For all $\varphi\in C^{n}(\huaA)$ and $x_i\in \Gamma(\huaA)$, $i=1,\dots,n+1$, the coboundary operator
$\delta_\huaA$ is given by
 \begin{align}
\delta_\huaA\varphi(x_1,\dots,x_{n+1})
 &=\sum_{i=1}^{n}(-1)^{i+1}a_\huaA(x_i)\varphi(x_1,\dots,\hat{x_i},\dots,x_{n+1})\nonumber
 \\
&\phantom{=}-\sum_{i=1}^{n}(-1)^{i+1}\varphi(x_1,\dots,\hat{x_i},\dots,x_n,x_i\ast_\huaA x_{n+1})\nonumber
\\
&\phantom{=}+\sum_{1\leq i<j\leq n}(-1)^{i+j}\varphi([x_i,x_j]_\huaA,x_1,\dots,\hat{x_i},\dots,\hat{x_j},\dots,x_{n+1}).
\label{LSCA cohomology}
\end{align}

Let $(\huaA,\ast_\huaA,a_\huaA)$ be a left-symmetric algebroid. Define
\begin{gather*}
\Sym^2(\huaA)=\{H\in \huaA\otimes \huaA| H(\alpha,\beta)=H(\beta,\alpha),\,
\forall\alpha,\beta\in \Gamma(\huaA^*)\}.
\end{gather*}
For any $H\in \Sym^2(\huaA)$, the bundle map $H^\sharp\colon \huaA^*\longrightarrow \huaA$ is given by $H^\sharp(\alpha)(\beta)=H(\alpha,\beta)$.
In \cite{lsb2}, the authors introduced $[ H,H]\in\Gamma\big({\wedge}^2\huaA \otimes \huaA\big) $ as follows
\begin{align}
[H,H](\alpha_1,\alpha_2,\alpha_3)&=a_\huaA(H^\sharp(\alpha_1))\langle H^\sharp(\alpha_2),\alpha_3\rangle-a_\huaA(H^\sharp(\alpha_2))\langle H^\sharp(\alpha_1),\alpha_3\rangle\nonumber
\\
&\phantom{=}+\langle \alpha_1,H^\sharp(\alpha_2)\ast_\huaA H^\sharp(\alpha_3)\rangle-\langle\alpha_2,H^\sharp(\alpha_1)\ast_\huaA H^\sharp(\alpha_3)\rangle
\nonumber
\\
&\phantom{=}-\langle \alpha_3,[H^\sharp(\alpha_1),
H^\sharp(\alpha_2)]_\huaA\rangle,\label{brac2}
\end{align}
for all $\alpha_1,\alpha_2,\alpha_3\in\Gamma(\huaA^*)$. Suppose that $H^\sharp\colon \huaA^*\longrightarrow \huaA$ is nondegenerate. Then $(H^\sharp)^{-1}\colon \huaA\longrightarrow \huaA^*$ is also a symmetric bundle map, which gives rise to an element, denoted by $H^{-1}$, in $\Sym^2(\huaA^*)$.

\begin{pro}[\cite{lsb2}]%\label{pro:equivelent}
Let $(\huaA,\ast_\huaA,a_\huaA)$ be a left-symmetric algebroid and $H\in \Sym^2(\huaA)$. If $H$ is nondegenerate, then
 $[ H,H]=0$ if and only if $\delta_\huaA (H^{-1})=0,$ i.e. $H^{-1}$ is a $2$-cocycle on the left-symmetric algebroid $\huaA.$

\end{pro}

Recall that a {\it pseudo-Hessian metric} $g$ is a
pseudo-Riemannian metric $g$ on a flat manifold $(M,\nabla)$ such
that $g$ can be locally expressed by
$g_{ij}=\frac{\partial^2\varphi}{\partial x^i\partial x^j},$ where
$\varphi\in\CWM$ and $\big\{x^1,\dots,x^n\big\}$ is an affine coordinate
system with respect to $\nabla$. Then the pair $(\nabla,g)$ is called a
pseudo-Hessian structure on $M$. A manifold $M$ with a
pseudo-Hessian structure $(\nabla,g)$ is called a {\it pseudo-Hessian manifold.}
 See~\cite{Geometry-of-Hessian-structures} for more
details about pseudo-Hessian manifolds.
 Let $(M,\nabla)$ be a flat manifold and $g$ a pseudo-Riemannian metric on $M$. Then $(M,\nabla,g)$ is a pseudo-Hessian manifold if and only if $\delta_{T_\nabla M} g=0$, where $\delta_{T_\nabla M}$ is the coboundary operator given by \eqref{LSCA cohomology} associated to the left-symmetric algebroid $T_\nabla M $ given in Example \ref{ex:main}.

Now we give the main structure studied in this section.

\begin{defi}
Let $(\huaA,\ast_\huaA,a_\huaA)$ be a left-symmetric algebroid.
 \begin{itemize}\itemsep=0pt
 \item[$(i)$] If $H\in \Sym^2(\huaA)$ satisfies $[ H,H]=0$, then $H$ is called a {\it Koszul--Vinberg structure} on the left-symmetric algebroid $\huaA$;
\item[$(ii)$] If $\frkB\in \Sym^2(\huaA^*)$ is nondegenerate and satisfies $\delta_\huaA\frkB=0$, then $\frkB$ is called a {\it pseudo-Hessian structure} on the left-symmetric algebroid $\huaA$.
 \end{itemize}\end{defi}

Let $(\huaA,\ast_\huaA,a_\huaA)$ be a left-symmetric algebroid, and $H\in \Sym^2(\huaA)$. Define
\begin{gather}\label{eq:multiplication-H}
\alpha\ast_{H^\sharp} \beta=\huaL_{H^\sharp(\alpha)}\beta-R_{H^\sharp(\beta)}\alpha-{\dM_\huaA} (H(\alpha,\beta)), \qquad\forall
\alpha,\beta\in\Gamma(\huaA^*),
\end{gather}
where $\huaL$ is the Lie derivation of the sub-adjacent Lie algebroid $\huaA^c$, $R$ and $\dM_\huaA$ are given by
\begin{gather*}
\langle R_x\alpha,y\rangle=-\langle\alpha,y\ast_\huaA x\rangle,\qquad
\dM_\huaA f(x)=a_\huaA(x)f,\qquad
\forall x,y\in\Gamma(\huaA),\quad f\in \CWM.
\end{gather*}

{\sloppy\begin{thm}[\cite{lsb2}]\label{thm:LSBi-H}
If $H$ is a Koszul--Vinberg structure on a left-symmetric algebroid $(\huaA,\ast_\huaA,a_\huaA)$, then $(\huaA^*,\ast_{H^\sharp},a_{H^\sharp}=a_\huaA\circ H^\sharp)$ is a left-symmetric algebroid, and $H^\sharp$ is a left-symmetric algebroid homomorphism from $(\huaA^*,\ast_{H^\sharp},a_{H^\sharp})$ to $(\huaA,\ast_\huaA,a_\huaA)$.
\end{thm}}

{\sloppy The sub-adjacent Lie algebroid of the left-symmetric algebroid $(\huaA^*,\ast_{H^\sharp},a_{H^\sharp})$ is $(\huaA^*,[\cdot,\cdot]_{H^\sharp},a_{H^\sharp})$, where $[\cdot,\cdot]_{H^\sharp}$ is given by
\begin{gather}\label{eq:commutator-H}
 [\alpha,\beta]_{H^\sharp}=L_{H^\sharp(\alpha)}\beta-L_{H^\sharp(\beta)}\alpha,\qquad
 \forall \alpha,\beta\in\Gamma(\huaA^*),
\end{gather}}\noindent
where $L$ is given by \eqref{eq:dualLR}.

\begin{pro}[\cite{lsb2}]\label{pro:morphism}
With the above notations,
for all $\alpha,\beta\in\Gamma(\huaA^*)$, we have
\begin{gather*}
%\label{homo2}
H^\sharp({[\alpha,\beta]}_{H^\sharp})-[H^\sharp(\alpha), H^\sharp(\beta)]_\huaA=[ H,H](\alpha,\beta,\cdot).
\end{gather*}
\end{pro}

Note that $L\colon \huaA\longrightarrow \frkD(\huaA^*)$ is a representation of the sub-adjacent Lie algebroid $\huaA^c$ on the dual bundle $\huaA^*$. Thus, by Proposition \ref{pro:morphism}, we have
\begin{pro}\label{pro:LSBi-H}
 $H$ is a Koszul--Vinberg structure on a left-symmetric algebroid $(\huaA,\ast_\huaA,a_\huaA)$ if and only if $H^\sharp\colon \huaA^*\longrightarrow \huaA $ is a relative Rota--Baxter operator on the \LP pair $(\huaA^c;L)$.
\end{pro}

By Theorem \ref{thm:graded Lie algebra} and Proposition \ref{pro:LSBi-H}, we have
\begin{lem}\label{lem:graded Lie algbra}
Let $(\huaA,\ast_\huaA,a_\huaA)$ be a left-symmetric algebroid and $H\in \Sym^2(\huaA)$.	
\begin{itemize}\itemsep=0pt
\item[$(i)$] $\big(\huaC^*(\huaA^*,\huaA):=\oplus_{k\geq0}\Gamma\big(\Hom\big({\wedge}^{k}\huaA^*,\huaA\big)\big),\Courant{\cdot,\cdot}\big)$ is a graded Lie algebra, where the bracket $\Courant{\cdot,\cdot}$ is given by \eqref{eq:graded Lie bracket}, in which $\rho=L$ is given by \eqref{eq:dualLR}.
\item[$(ii)$] $H$ is a Koszul--Vinberg structure on the left-symmetric algebroid if and only if $H^\sharp$ is a~Maurer--Cartan element of the graded Lie algebra $(\huaC^*(\huaA^*,\huaA),\Courant{\cdot,\cdot})$.
\end{itemize}
\end{lem}

For $k\geq0$, define $\Psi\colon \Gamma\big({\wedge}^{k}\huaA\otimes \huaA\big)\longrightarrow\huaC^k(\huaA^*,\huaA) $ by
\begin{gather}\label{eq:relation-coboundary} \langle\Psi(\varphi)(\alpha_1,\dots,\alpha_k),\alpha_{k+1}\rangle =\langle\varphi,\alpha_1\wedge\cdots\wedge\alpha_k\otimes\alpha_{k+1}\rangle,\qquad
\forall \alpha_1,\dots,\alpha_{k+1}\in\Gamma(\huaA^*),
\end{gather}
and $\Upsilon\colon \huaC^k(\huaA^*,\huaA) \longrightarrow \Gamma\big({\wedge}^{k}\huaA\otimes \huaA\big)$ by
\begin{gather*}%\label{eq:defiUpsilon}
\langle \Upsilon(P),\alpha_1\wedge\cdots\wedge\alpha_k\otimes\alpha_{k+1}\rangle=\langle P(\alpha_1,\dots,\alpha_k),\alpha_{k+1}\rangle,\qquad
\forall \alpha_1,\dots,\alpha_{k+1}\in\Gamma(\huaA^*).
\end{gather*}
Obviously we have $\Psi\circ\Upsilon={\Id},~~\Upsilon\circ\Psi={\Id}$.

By Lemma \ref{lem:graded Lie algbra}, we have
\begin{thm}
Let $(\huaA,\ast_\huaA,a_\huaA)$ be a left-symmetric algebroid. Then, there is a graded Lie bracket $\Courant{\cdot,\cdot}_{\rm KV}\colon \Gamma\big({\wedge}^{k}\huaA\otimes \huaA\big)\times \Gamma\big({\wedge}^{l}\huaA\otimes \huaA\big)\longrightarrow \Gamma\big({\wedge}^{k+l}\huaA\otimes \huaA\big)$ on the graded vector space $C_{\rm KV}^{*}(\huaA^*):=\oplus_{k\ge 1}C_{\rm KV}^{k}(\huaA^*)$ with $C_{\rm KV}^{k}(\huaA^*):=\Gamma\big({\wedge}^{k-1}\huaA\otimes \huaA\big)$ given by
\begin{gather*}
\Courant{\varphi,\phi}_{\rm KV}:=\Upsilon\Courant{\Psi(\varphi),\Psi(\phi)},\qquad
\forall \varphi\in\Gamma\big({\wedge}^{k}\huaA\otimes \huaA\big),\quad\phi\in\Gamma\big({\wedge}^{l}\huaA\otimes \huaA\big).
\end{gather*}

Furthermore, $H\in \Sym^2(\huaA)$ is a Koszul--Vinberg structure on the left-symmetric algebroid~$\huaA$ if and only if $H$ is a Maurer--Cartan element of the graded Lie algebra $(C_{\rm KV}^{*}(\huaA^*),\Courant{\cdot,\cdot}_{\rm KV})$. More precisely, we have
 \begin{gather*}
\Courant{H,H}_{\rm KV}(\alpha_1,\alpha_2,\alpha_3)=2[H,H](\alpha_1,\alpha_2,\alpha_3),\qquad \forall \alpha_1,\alpha_2,\alpha_3\in\Gamma(\huaA^*),
 \end{gather*}
 where $[H,H]$ is given by \eqref{brac2}.
\end{thm}

\begin{rmk}
We~characterize a Koszul--Vinberg structure on a left-symmetric algebroid $\huaA$ as a Maurer--Cartan element of the graded Lie algebra $(C_{\rm KV}^{*}(\huaA^*),\Courant{\cdot,\cdot}_{\rm KV})$. This is parallel to the fact that a Poisson structure is a Maurer--Cartan element of the graded Lie algebra given by the Schouten--Nijenhuis bracket of multi-vector fields.
\end{rmk}

\subsection{Cohomologies and deformations of Koszul--Vinberg structures}

Let $H\in \Sym^2(\huaA)$ be a Koszul--Vinberg structure on a left-symmetric algebroid $(\huaA,\ast_\huaA,a_\huaA)$. Define $\delta_{\huaA^*}\colon C_{\rm KV}^{k}(\huaA^*)\longrightarrow
C_{\rm KV}^{k+1}(\huaA^*)$ by
\begin{gather*}
\delta_{\huaA^*}\varphi =(-1)^{k-1}\Courant{H,\varphi}_{\rm KV},\qquad
\forall \varphi\in C_{\rm KV}^{k}(\huaA^*).
\end{gather*}
By the graded Jacobi identity, we have $\delta_{\huaA^*}\circ \delta_{\huaA^*}=0$. Thus $(C_{\rm KV}^{*}(\huaA^*),\delta_{\huaA^*})$ is a cochain comp\-lex. Denote by $H_{\rm KV}^k(\huaA^*)$ the $k$-th cohomology group, called the {\it $k$-th cohomology group of the Koszul--Vinberg structure $H$}.

Furthermore, we have
\begin{pro}
For $\varphi\in C_{\rm KV}^{k}(\huaA^*)$, we have
 \begin{align*}%\label{eq:coboundary of A*}
\delta_{\huaA^*}\varphi(\alpha_1,\dots,\alpha_{k+1})
 &=\sum_{i=1}^{k}(-1)^{i+1}a_{H^\sharp} (\alpha_i)\varphi(\alpha_1,\dots,\hat{\alpha_i},\dots,\alpha_{k+1})
 \\
 &\phantom{=}-\sum_{i=1}^{k}(-1)^{i+1}\varphi (\alpha_1,\dots,\hat{\alpha_i},\dots,\alpha_k,\alpha_i\ast_{H^\sharp} \alpha_{k+1})
 \\
 &\phantom{=}+\sum_{1\leq i<j\leq k}(-1)^{i+j}\varphi([\alpha_i,\alpha_j]_{H^\sharp},\alpha_1, \dots,\hat{\alpha_i},\dots,\hat{\alpha_j},\dots,\alpha_{k+1}),
\end{align*}
where $\alpha_1,\dots,\alpha_{k+1}\in\Gamma(\huaA^*)$, $\ast_{H^\sharp}$ is given by \eqref{eq:multiplication-H} and $[\cdot,\cdot]_{H^\sharp}$ is given by \eqref{eq:commutator-H}.
\end{pro}

\begin{proof}
 It~follows by a direct calculation.
\end{proof}
\begin{rmk}
 Note that this coboundary operator $\delta_{\huaA^*}$ is just the coboundary operator given by \eqref{LSCA cohomology} associated to the left-symmetric algebroid $(\huaA^*,\ast_{H^\sharp},a_{H^\sharp})$ in Theorem \ref{thm:LSBi-H}.
\end{rmk}

By Corollary \ref{cor:construction left-symmetrci algebroid} and Proposition \ref{pro:LSBi-H}, we have
\begin{pro}
 {\sloppy Let $H$ be a Koszul--Vinberg structure on a left-symmetric algebroid $(\huaA,\ast_\huaA,a_\huaA)$. Then $(\huaA^*,\cdot_{H^\sharp},a_{H^\sharp}=a_\huaA\circ H^\sharp)$ is a left-symmetric algebroid, where $\cdot_{H^\sharp}$ is given~by
 \begin{gather*}
 \alpha\cdot_{H^\sharp}\beta=L_{H^\sharp(\alpha)}\beta,\qquad
 \forall \alpha,\beta\in\Gamma(\huaA^*).
 \end{gather*}}
\end{pro}

\begin{rmk}
 The left-symmetric algebroids $(\huaA^*,\cdot_{H^\sharp},a_{H^\sharp})$ and $(\huaA^*,\ast_{H^\sharp},a_{H^\sharp})$ have the same sub-adjacent Lie algebroid $(\huaA^*,[\cdot,\cdot]_{H^\sharp},a_{H^\sharp})$.
\end{rmk}

By Lemma \ref{ex:MC-ten}, we have
\begin{pro}%\label{pro:KV-structure}
{\sloppy Let $H$ be a Koszul--Vinberg structure on a left-symmetric algebroid $(\huaA,\ast_\huaA,a_\huaA)$.
 Then
 \begin{gather}\label{eq:rep-coadjoint}
\varrho\colon\ \huaA^*\longrightarrow \frkD(\huaA),\qquad \varrho(\alpha)(x)=[H^\sharp(\alpha),x]_{\huaA}+H^\sharp(L_x\alpha ),\qquad \forall x\in\Gamma{(\huaA)},\alpha\in\Gamma(\huaA^*)
\end{gather}}\noindent
 is a representation of the sub-adjacent Lie algebroid $(\huaA^*,[\cdot,\cdot]_{H^\sharp},a_{H^\sharp})$ on the vector bundle $\huaA$.
\end{pro}

\begin{rmk}%\label{rmk:KV-structure}
The representation $\varrho$ given by \eqref{eq:rep-coadjoint} is exactly the dual representation of the left multiplication operation of the left-symmetric algebroid $(\huaA^*,\ast_{H^\sharp},a_{H^\sharp})$. More precisely, let us denote by $\frkL\colon \huaA^*\longrightarrow \frkD(\huaA^*)$ the left multiplication operation of the left-symmetric algebroid $(\huaA^*,\ast_{H^\sharp},a_{H^\sharp})$, then we have
\begin{align*}
\langle\frkL_\alpha x,\beta\rangle&=a_{H^\sharp}(\alpha)\langle x,\beta\rangle-\langle x,\alpha \ast_{H^\sharp} \beta\rangle
\\
&=a_{H^\sharp}(\alpha)\langle x,\beta\rangle-\langle x,\huaL_{H^\sharp(\alpha)}\beta-R_{H^\sharp(\beta)}\alpha-{\dM_\huaA} (H(\alpha,\beta))\rangle
\\
&=a_{H^\sharp}(\alpha)\langle x,\beta\rangle\!-\!a_\huaA( H^\sharp(\alpha))\langle x,\beta\rangle\!+\![H^\sharp(\alpha),x]_{\huaA}\!-\!\langle\alpha,x\ast_\huaA H^\sharp(\beta)\rangle
\!+\!a_\huaA(x)H(\alpha,\beta)
\\
&=\langle[H^\sharp(\alpha),x]_{\huaA}+H^\sharp(L_x\alpha ),\beta\rangle
\\
&=\langle\varrho(\alpha)(x),\beta\rangle.
\end{align*}
Thus we have $\frkL_\alpha x=\varrho(\alpha)(x)$.
\end{rmk}

Let $H$ be a Koszul--Vinberg structure on a left-symmetric algebroid $(\huaA,\ast_\huaA,a_\huaA)$. By Theo\-rem~\ref{pro:coboundary operator relation}, for $P\in\huaC^k(\huaA^*,\huaA)$ and $\alpha_1,\dots,\alpha_{k+1}\in\Gamma{(\huaA^*)}$, the coboundary operator $\dM_{H^\sharp}\colon \huaC^k(\huaA^*,\huaA)\allowbreak\longrightarrow \huaC^{k+1}(\huaA^*,\huaA)$ of the relative Rota--Baxter operator $H^\sharp$ is given by
\begin{align*}%\label{eq:coboundary of A*A}
\dM_{H^\sharp}P(\alpha_1,\dots,\alpha_{k+1})
&=\sum_{i=1}^{k+1}(-1)^{i+1}[H^\sharp(\alpha_i),P(\alpha_1,\alpha_2,
\dots,\hat{\alpha_i},\dots,\alpha_{k+1})]_{\huaA}
\\
&\phantom{=}+\sum_{i=1}^{k+1}(-1)^{i+1}H^\sharp\big(L_{P(\alpha_1,\alpha_2, \dots,\hat{\alpha_i},\dots,\alpha_{k+1})}\alpha_i\big)
\\
&\phantom{=}+\sum_{1\leq i<j\leq {k+1}}(-1)^{i+j}P([\alpha_i,\alpha_j]_{H^\sharp},\alpha_1, \dots,\hat{\alpha_i},\dots,\hat{\alpha_j},\dots,\alpha_{k+1}).
\end{align*}
Denote by $H^k(\huaA^*,\huaA)$ the $k$-th cohomology group, called the {\it $k$-th cohomology group of the relative Rota--Baxter operator $H^\sharp$}.

\begin{pro}\label{pro:isomorphism}
 With the above notations, the map $\Psi$ defined by \eqref{eq:relation-coboundary} is a cochain isomorphism between cochain complexes $(C_{\rm KV}^{*}(\huaA^*),\delta_{\huaA^*})$ and $(\huaC^*(\huaA^*,\huaA),\dM_{H^\sharp})$, i.e., we have the following commutative diagram:
 \begin{gather*}
\xymatrix{
\cdots
\longrightarrow C_{\rm KV}^{k+1}(\huaA^*) \ar[d]^{\Psi} \ar[r]^{\quad\delta_{\huaA^*}} & C_{\rm KV}^{k+2}(\huaA^*) \ar[d]^{\Psi} \ar[r] & \cdots \\
\cdots\longrightarrow \huaC^k(\huaA^*,\huaA) \ar[r]^{\quad \dM_{H^\sharp}} &\huaC^{k+1}(\huaA^*,\huaA)\ar[r]& \cdots.}
\end{gather*}
 Consequently, $\Psi$ induces an isomorphism map $\Psi_\ast$ between the corresponding cohomology groups.
\end{pro}
\begin{proof}
It~is straightforward to see that $\Psi$ is a graded Lie algebra isomorphism between the graded Lie algebra $(\huaC^*(\huaA^*,\huaA),\Courant{\cdot,\cdot})$ and $(C_{\rm KV}^{*}(\huaA^*),\Courant{\cdot,\cdot}_{\rm KV})$. Thus for any $P\in C_{\rm KV}^{k+1}(\huaA^*)$, we~have
\begin{gather*}
\Psi(\delta_{\huaA^*}P)=\Psi((-1)^k\Courant{H,P}_{\rm KV})=(-1)^k\Courant{\Psi(H),\Psi(P)}=\dM_{H^\sharp}\Psi(P),
\end{gather*}
which implies that $\dM_{H^\sharp}\circ \Psi =\Psi\circ\delta_{\huaA^*}$, i.e., the map $\Psi$ is a cochain map between cochain complexes $(C_{\rm KV}^{*}(\huaA^*),\delta_{\huaA^*})$ and $(\huaC^*(\huaA^*,\huaA),\dM_{H^\sharp})$.
Consequently, for any $k\geq0$, $\Psi$ induces an isomorphism between the corresponding cohomology groups.
\end{proof}

Now we introduce a new cochain complex, whose cohomology groups control deformations of Koszul--Vinberg structures. Let $H$ be a Koszul--Vinberg structure on a left-symmetric algebroid $(\huaA,\ast_\huaA,a_\huaA)$. For all $\alpha_1,\alpha_2,\alpha_3\in \Gamma(\huaA^*)$, define
\begin{gather*}
	\tilde{\huaC}_{\rm KV}^1(\huaA^*)=\big\{x \in 	{\huaC}_{\rm KV}^1(\huaA^*)\mid H(R_x \alpha_1,\alpha_2)=H(\alpha_1,R_x\alpha_2)\big\},
\\
	\tilde{\huaC}_{\rm KV}^2(\huaA^*)=\big\{\varphi\in {\huaC}_{\rm KV}^2(\huaA^*)\mid \varphi(\alpha_1,\alpha_2)=\varphi(\alpha_2,\alpha_1)\big\},
\\
	\tilde{\huaC}_{\rm KV}^3(\huaA^*)=\big\{\varphi\in {\huaC}_{\rm KV}^3(\huaA^*)\mid \varphi(\alpha_1,\alpha_2,\alpha_3)+c.p.=0\big\},
\\
	\tilde{\huaC}_{\rm KV}^k(\huaA^*)={\huaC}_{\rm KV}^k(\huaA^*),\qquad k\geq 4.
\end{gather*}
It~is straightforward to verify that the cochain complex $\big(\tilde{C}_{\rm KV}^{*}(\huaA^*),\delta_{\huaA^*}\big)$ is a subcomplex of the cochain complex $(C_{\rm KV}^{*}(\huaA^*),\delta_{\huaA^*})$. Denote by $\tilde{H}_{\rm KV}^k(\huaA^*)$ the $k$-th cohomology group.

\begin{defi}
{\sloppy Let $H$ be a Koszul--Vinberg structure on a left-symmetric algebroid $(\huaA,\ast_\huaA,a_\huaA)$. A {\it formal deformation} of the Koszul--Vinberg structure $H$ is a formal power series
\begin{gather*}
H_t=\sum_{i=0}^{+\infty}\huaH_it^i\in\Sym^2(\huaA)[[t]]
\end{gather*}}\noindent
 such that $H_t$ is a Koszul--Vinberg structure on the left-symmetric algebroid $(\huaA\otimes \Real[[t]],\ast_\huaA,a_\huaA)$ and $\huaH_0=H$.
\end{defi}
Note that $H_t$ is a formal deformation of the Koszul--Vinberg structure $H$ if and only if $H_t^\sharp$ is a formal deformation of the relative Rota--Baxter operator $H^\sharp$ on the \LP pair $(\huaA^c;L)$.

{\sloppy\begin{defi}
Let $H$ be a Koszul--Vinberg structure on a left-symmetric algebroid $(\huaA,\ast_\huaA,a_\huaA)$. If $H_{(n)}=\sum_{i=0}^{n}\huaH_it^i$ with $\huaH_0=H$, $\huaH_i\in\Sym^2(\huaA)$, $i=1,\dots,n$ is a Koszul--Vinberg structure on the left-symmetric algebroid $\big(\huaA\otimes \Real[[t]]/\big(t^{n+1}\big),\ast_\huaA,a_\huaA\big)$, we say that $H_{(n)}$ is an {\it order~$n$ deformation} of the Koszul--Vinberg structure $H$. Furthermore, if there exists an element $\huaH_{n+1}\in\Sym^2(\huaA)$ such that $H_{(n+1)}=H_{(n)}+t^{n+1}\huaH_{n+1}$ is an order $n$ deformation of the Koszul--Vinberg structure $H$, we say that $H_{(n)}$ is {\it extendable}.
\end{defi}}\noindent
 We~call an order $1$ deformation of the Koszul--Vinberg structure $H$ on a left-symmetric algebroid $(\huaA,\ast_\huaA,a_\huaA)$ an {\it infinitesimal deformation} of the Koszul--Vinberg structure $H$.

It~is not hard to check that $H_{(n)}$ is an order $n$ deformation of the Koszul--Vinberg structure~$H$ if and only if $H_{(n)}^\sharp$ is an order $n$ deformation of the relative Rota--Baxter operator $H^\sharp$ on the \LP pair $(\huaA^c;L)$.

\begin{defi}
{\sloppy
Let $H$ be a Koszul--Vinberg structure on a left-symmetric algebroid $(\huaA,\ast_\huaA,a_\huaA)$. Two order $n$ deformations $H_t$ and $H'_t$ of $H$ are said to be {\it equivalent} if there exists a formal series $\huaX_t=\sum_{i=1}^{+\infty}x_it^i,~ x_i\in \Gamma(\huaA	)$ such that
\begin{gather*}%\label{eq:equivalent 2}
{\rm exp}(\ad_{\huaX_t})H_t=H'_{t} \mbox{ modulo } t^{n+1},
\end{gather*}}\noindent
where ${\rm exp}$ denotes the exponential series and\vspace{-.5ex}
\begin{gather*}
\ad^k_{\huaX_t}H_t=\big[\hspace{-3pt}\big[\huaX_t,\big[\hspace{-3pt}\big[\huaX_t,\dots,[\hspace{-1.5pt}[\huaX_t,\stackrel{}H_t]\hspace{-1.5pt}]_{\rm KV},\stackrel{k}{\dots}\big]\hspace{-3pt}\big]_{\rm KV}\big]\hspace{-3pt}\big]_{\rm KV}.
\end{gather*}
An order $n$ deformation $H_t$ of $H$ is called {\it trivial} if $H_t$ is equivalent to $H$.
\end{defi}

{\sloppy\begin{pro}\label{pro:equivalent describe}
 Let $H$ be a Koszul--Vinberg structure on a left-symmetric algebroid $(\huaA,\ast_\huaA,a_\huaA)$ and $H_t,H'_t\in\Sym^2(\huaA)[[t]]$. Two order $n$ deformations $H_t$ and $H'_t$ of the Koszul--Vinberg structure $H$ are equivalent if and only if the two order $n$ deformations $H^\sharp_t$ and $(H')^\sharp_t$ of the relative Rota--Baxter operator $H^\sharp$ on the \LP pair $(\huaA^c;L)$ are equivalent.
\end{pro}}

\begin{proof}
It~follows from that $\Psi$ defined by \eqref{eq:relation-coboundary} is a graded Lie algebra isomorphism between the graded Lie algebra $(\huaC^*(\huaA^*,\huaA),\Courant{\cdot,\cdot})$ and $(C_{\rm KV}^{*}(\huaA^*),\Courant{\cdot,\cdot}_{\rm KV})$.
\end{proof}

{\sloppy\begin{pro}
Let $H$ be a Koszul--Vinberg structure on a left-symmetric algebroid $(\huaA,\ast_\huaA,a_\huaA)$. Then there is a one-to-one correspondence between equivalence classes of infinitesimal deformations of the Koszul--Vinberg structure $H$ and the second cohomology group $\tilde{H}_{\rm KV}^2(\huaA^*)$.
\end{pro}}

\begin{proof}
Assume that $H_t$ and $H'_t$ are equivalent infinitesimal deformations of the Koszul--Vinberg structure $H$. By Theorem \ref{pro:equivalence class of O} and Proposition \ref{pro:equivalent describe}, there exists an element $x\in\Gamma(\huaA)$ such that\vspace{-.5ex}
\begin{gather*}
\huaH_1'- \huaH_1=\delta_{\huaA^* }x.
\end{gather*}
Since $\huaH'_1$ and $\huaH_1$ are symmetric, for all $\alpha_1,\alpha_2\in \Gamma(\huaA^*)$, we have\vspace{-.5ex}
\begin{gather*}
\delta_{\huaA^* }x(\alpha_1,\alpha_2)=\delta_{\huaA^* }x(\alpha_2,\alpha_1),
\end{gather*}
which implies that $H(R_x \alpha_1,\alpha_2)=H(\alpha_1,R_x\alpha_2)$, i.e., $x\in 	\tilde{\huaC}_{\rm KV}^1(\huaA^*)$. Thus $\huaH_1'$ and $\huaH_1$ are in the same cohomology class of $\tilde{H}_{\rm KV}^2(\huaA^*)$.

The converse can be proved similarly. We~omit the details.\vspace{-.5ex}
\end{proof}

Similarly to Proposition \ref{pro:trivial of O-operator}, we have\vspace{-.5ex}
{\sloppy\begin{pro}
Let $H$ be a Koszul--Vinberg structure on a left-symmetric algebroid $(\huaA,\ast_\huaA,a_\huaA)$ such that $\tilde{H}_{\rm KV}^2(\huaA^*)=0$. Then all infinitesimal deformations of the Koszul--Vinberg structure $H$ are trivial.
\end{pro}}

\begin{thm}
Let $H$ be a Koszul--Vinberg structure on a left-symmetric algebroid $(\huaA,\ast_\huaA,a_\huaA)$. Let $H_{(n)}=\sum_{i=0}^{n}\huaH_it^i$ be an order $n$ deformation of $H$. Define\vspace{-.5ex}
 \begin{gather}\label{eq:3cocycle}
 \Theta=\half\sum_{\substack{i+j=n+1\\ i,j\geq1}}\Courant{\huaH_i,\huaH_j}_{\rm KV}.
 \end{gather}
 Then the $3$-cochain $\Theta$ is closed, i.e., $\delta_{\huaA^*}\Theta=0$. Furthermore, $H_{(n)}$ is {extendable} if and only if the cohomology class $[\Theta]$ in $\tilde{H}_{\rm KV}^3(\huaA^*)$ is trivial.
\end{thm}
\begin{proof}
	For any $ \alpha_1,\alpha_2,\alpha_3\in\Gamma(\huaA^*)$ and $i,j\geq 1$, we have
\begin{align*}%\label{brac3}
\Courant{\huaH_i,\huaH_j}_{\rm KV}(\alpha_1,\alpha_2,\alpha_3)
&=a_\huaA(\huaH^\sharp_i(\alpha_1))\langle \huaH^\sharp_j(\alpha_2),\alpha_3\rangle+a_\huaA(\huaH^\sharp_j(\alpha_1))\langle \huaH^\sharp_i(\alpha_2),\alpha_3\rangle
\\
&\phantom{=}-a_\huaA(\huaH^\sharp_i(\alpha_2))\langle \huaH^\sharp_j(\alpha_1),\alpha_3\rangle
-a_\huaA(\huaH^\sharp_j(\alpha_2))\langle \huaH^\sharp_i(\alpha_1),\alpha_3\rangle
\\
&\phantom{=}+\langle \alpha_1,\huaH^\sharp_i(\alpha_2)\ast_\huaA \huaH^\sharp_j(\alpha_3)\rangle +\langle \alpha_1,\huaH^\sharp_j(\alpha_2)\ast_\huaA \huaH^\sharp_i(\alpha_3)\rangle \\
&\phantom{=}-\langle\alpha_2,\huaH^\sharp_i(\alpha_1)\ast_\huaA \huaH^\sharp_j(\alpha_3)\rangle-\langle\alpha_2,\huaH^\sharp_j(\alpha_1)\ast_\huaA \huaH^\sharp_i(\alpha_3)\rangle
\\
&\phantom{=}-\langle \alpha_3,[\huaH^\sharp_i(\alpha_1),
\huaH^\sharp_j(\alpha_2)]_\huaA\rangle
-\langle \alpha_3,[\huaH^\sharp_j(\alpha_1),
\huaH^\sharp_i(\alpha_2)]_\huaA\rangle.
\end{align*}
It~is straightforward to	check that\vspace{-.5ex}
\begin{gather*}
\Courant{\huaH_i,\huaH_j}_{\rm KV}(\alpha_1,\alpha_2,\alpha_3)+ \Courant{\huaH_i,\huaH_j}_{\rm KV}(\alpha_3,\alpha_1,\alpha_2)+\Courant{\huaH_i,\huaH_j}_{\rm KV}(\alpha_2,\alpha_3,\alpha_1)=0,
\end{gather*}
which implies that $\Theta$ defined by \eqref{eq:3cocycle} is in $\tilde{\huaC}_{\rm KV}^3(\huaA^*)$. By Theorem \ref{thm:extendable of O operrator} and Proposition \ref{pro:isomorphism}, the $3$-cochain $\Theta$ is closed. The rest follows directly from the fact that this deformation problem is controlled by the differential graded Lie algebra $(C_{\rm KV}^{*}(\huaA^*),\Courant{\cdot,\cdot}_{\rm KV},\Courant{H,\cdot}_{\rm KV})$. We~omit the details.\vspace{-1ex}
	\end{proof}	

\subsection*{Acknowledgements}

This research was supported by the National Key Research and Development Program of China (2021YFA1002000), the National Natural Science Foundation of China (11901501, 11922110), the China Postdoctoral Science Foundation (2021M700750) and the Fundamental Research Funds for the Central Universities (2412022QD033). We~give our warmest thanks to the referees for very useful comments that improve the paper.\vspace{-1ex}

\pdfbookmark[1]{References}{ref}
\LastPageEnding

\end{document}